%% file: unlocking_optimal_batch_sizes_arxiv_ver1.tex
\documentclass{article}

\usepackage[utf8]{inputenc} % allow utf-8 input
\usepackage[T1]{fontenc}
\usepackage{lmodern}
\usepackage{hyperref}       % hyperlinks
\usepackage{url}            % simple URL typesetting
\usepackage{booktabs}       % professional-quality tables
\usepackage{amsfonts}       % blackboard math symbols
\usepackage{nicefrac}       % compact symbols for 1/2, etc.
\usepackage{microtype}      % microtypography
\usepackage{xcolor}         % colors

\usepackage{bookmark}
\usepackage{bibentry}
\usepackage{bm}

\newcommand{\tiid}{i.i.d.\ }

\newcommand{\Kurt}{\operatorname{Kurt}}
\usepackage{amsmath}

\usepackage{algorithm, algorithmicx, algpseudocode}
\usepackage{graphicx}
\input{commands}

\usepackage{ifthen}
\newcommand{\showtodos}{show them}  % comment this line if you don't 
\usepackage{halloweenmath}
        {%
            \ifthenelse{\isundefined{\showtodos}}%
                    {\expandafter\comment}%
                    {\noindent\color{blue}$\spadesuit$}%
                    }%
         {%
            \ifthenelse{\isundefined{\showtodos}}%
                    {\expandafter\endcomment}%
                    {\begin{flushright}
                    $\spadesuit$
                    \end{flushright}\color{black}}%
          }

\newcounter{alphasect}
\def\alphainsection{0}

\let\oldsection=\section
\def\section{%
	\ifnum\alphainsection=1%
	\addtocounter{alphasect}{1}
	\fi%
	\oldsection}%
\renewcommand\thesection{%
	\ifnum\alphainsection=1% 
	\Alph{alphasect}%
	\else
	\arabic{section}%
	\fi%
}%

\newenvironment{alphasection}{%
	\ifnum\alphainsection=1%
	\errhelp={Let other blocks end at the beginning of the next block.}
	\errmessage{Nested Alpha section not allowed}
	\fi%
	\setcounter{alphasect}{0}
	\def\alphainsection{1}
}{%
	\setcounter{alphasect}{0}
	\def\alphainsection{0}
}%
\newcommand{\bx}{\bm{x}}
\newcommand{\by}{\bm{y}}
\newcommand{\bz}{\bm{z}}

\newcommand{\GSGD}{\chi}

\newcommand{\nrmp}[2]{\nrm{#1}{#2}}

\newcommand{\nrmps}[2]{\nrm{#1}{*#2}}
\newcommand{\nrmpsp}[3]{\nrm{#1}{*#2}^{#3}}

\begin{document}
\title{Unlocking optimal batch size schedules using continuous-time control and perturbation theory}
\author{Stefan Perko\thanks{Institute for Mathematics, Friedrich-Schiller-University Jena, 07737 Jena, Germany.  Email: stefan.perko@uni-jena.de}}
\maketitle

\begin{abstract}
Stochastic Gradient Descent (SGD) and its variants are almost universally used to train neural networks and to fit a variety of other parametric models. An important hyperparameter in this context is the batch size, which determines how many samples are processed before an update of the parameters occurs. Previous studies have demonstrated the benefits of using variable batch sizes. In this work, we will theoretically derive optimal batch size schedules for SGD and similar algorithms, up to an error that is quadratic in the learning rate. To achieve this, we approximate the discrete process of parameter updates using a family of stochastic differential equations indexed by the learning rate. To better handle the state-dependent diffusion coefficient, we further expand the solution of this family into a series with respect to the learning rate. Using this setup, we derive a continuous-time optimal batch size schedule for a large family of diffusion coefficients and then apply the results in the setting of linear regression.
\end{abstract}

\section{Introduction}
Let $d \in \N$ and consider a family of risk functions 
\[R : \R^d \times \cZ \to [0,\infty), (\thet, z) \mapsto R_z(\thet)\]
and a probability measure $\nu$ on $\cZ$. The risk minimization task associated with $(R,\nu)$ is
\begin{equation}
\label{eq:riskMin}
\min_{\thet\in \R^d} \cR(\thet),
\end{equation}
where $\cR(\thet) = \E_{z\sim \nu}[R_z(\thet)]$.
\begin{comment}
For example if $\cZ = \R^m \times \R$, with $m\in \N$, we are given a loss function $\ell : \R\times \R \to [0,\infty)$ and a parameterized model $(f_\thet : \R^m \to \R)_{\thet \in \R^d}$, we may consider the associated risk
\[R_{x,y}(\thet) = \ell(f_\thet(x),y), \thet \in \R^d\]
due to a data point $(x,y) \in \cZ$.
Suppose further that we are given data
\[\cD = \set{(x_0,y_0),\dots, (x_{N-1}, y_{N-1})}\subseteq \R^m\times \R\]
and $\nu$ is the empirical measure associated with $\cD$, given by the mean of Dirac measures
\[\nu = \frac 1 N \sum_{n=0}^{N-1} \idK_{(x_n,y_n)}.\]
Then problem \eqref{eq:riskMin} is the task of minimizing the \emph{empirical risk} (a.k.a. training error)
\begin{equation}
\label{eq:empRiskMin}
\cR(\thet) = \frac1N \sum_{n=0}^{N-1} \ell(f_\thet(x_n),y_n).
\end{equation}
Parameterized models include regression models, such as linear regression, logistic regression and neural networks. Possible loss functions include quadratic loss and cross-entropy loss.

Rather than using a (possibly random) empirical measure, one can also consider directly the \emph{population} $\nu$, i.e.\ the true distribution our data is drawn from. In this case $\cR$ is called the \emph{population risk}. Since $\nu$ cannot be observed in practice, minimization of the empirical risk \eqref{eq:empRiskMin} is usually studied as a proxy problem.
\end{comment}
To solve \eqref{eq:riskMin} one frequently uses a one-step method of the form
\begin{equation}
\label{eq:mbSGD}
\chi_{n+1}^h = \chi_n^h + h f_{nh}(\chi_n^h),
\end{equation}
for a learning rate $h\in (0,1)$, where $(f_t^h)_{t\geq 0, h\in (0,1)}$ is a family of independent random functions $\R^d\to \R^d$.

For convenience we use a continuous time point to index $f$, thereby viewing \eqref{eq:mbSGD} at the time points $n = 0, \dots, \floor{T/h}$ as a (stochastic) discretization of the following ODE
\begin{equation}
dX_t^0 = \E f_t(X_t^0)\,dt, \quad t\in [0,T],
\end{equation}
for a given continuous time horizon $T > 0$.
%While the continuous-time perspective is especially important to prove our main result Theorem \ref{thm:optBS}, we occasionally reference it before.

We are interested in studying a version of \eqref{eq:mbSGD} called \emph{mini-batch SGD}.
To this end, fix an \tiid sequence $\bz_0, \bz_1,\dots$, with $\bz_0 \sim \nu$, and a sequence of \emph{batch sizes} $(B_n)_{n\in \N}$. Consider the sequence of \emph{batches}
\[\cB_1 = \set{\bz_0, \dots, \bz_{B_1}}, \cB_2 = \set{\bz_{B_1 + 1}, \dots, \bz_{B_1 + B_2}}, \dots.\]
Then mini-batch SGD, with batch sizes $(B_n)_{n\in \N}$, uses the sequence of estimators
\begin{equation}
	\label{eq:mbgrads}
	f_{nh}(\thet) = - \frac{1}{B_n} \sum_{z\in \cB_n} \nabla R_z(\thet),\quad n\in \N
\end{equation}
Assuming $\E$ and $\nabla$ commute, we have $\E f_{nh}(\thet) = -\nabla \cR(\thet)$. Further, the covariance matrix of $f_{nh}$ is given by
\begin{align*}
	\Cov[f_{nh}(\thet)] =& \frac 1 {B_n} \Si(\thet),
\end{align*}
where
\[\Si(\thet) := \Cov_{z\sim \nu}[\nabla R_z(\thet)].\]
for all $\thet\in \R^d$. We can identify the sequence of \emph{inverse} batch sizes as a \emph{volatility control} $\al$, i.e.\ $\al_{nh} = \frac{1}{B_n}$. Since batch sizes are bounded below by $1$, we have the natural bounds $0 \leq \al \leq 1$.

\begin{comment}
In the case of plain SGD we want to select $f$, such that
\[\E[f_t(\thet)] = - u_t \nabla \cR(\thet), \Var[f_{t}(\thet)] = \al_t \Si(\thet).\]
Here, $u$ takes the role of a learning rate schedule and $\al$ is a volatility control. In the case of constant learning rates $u = 1$ we can identify $\al$ with the inverse (time-dependent) \emph{batch size}. Further, if the batch size is $1$, then $\Si$ can be interpreted as the covariance matrix of the sampled gradients, that is it describes the noise inherent to the SGD algorithm.

To make the connection of $\al$ to batch sizes explicit, fix a measure $\nu$ on $\R^d\times \R$ (either empirical or population), an \tiid{} sequence of data points $(x_k, y_k)_{k\in \N}$, each with distribution $\nu$, and a sequence of batch sizes $(B_n)_{n\in \N}$. Let the $n$-th batch be the $B_n$-element set given by
\[\cB_n = \set{(x_{1 + \sum_{j=0}^{n-1} B_j}, y_{1 + \sum_{j=0}^{n-1} B_j}), \dots, (x_{\sum_{j=0}^{n} B_j}, y_{\sum_{j=0}^n B_j})}.\] 
In the empirical case this means that in the $n$-th iteration, we draw $B_n$-data points with replacement from our data set $\cD$. In the population case we draw $B_n$-data points from the population $\nu$ directly.
\end{comment}

%Our main result pertains to the optimal volatility schedule, up to an error or order $h^2$. However, note that we require volatility schedules to be Lipschitz, while inverse batch sizes are inverse integer valued, hence merely of bounded variation.  Therefore, we cannot directly transfer optimal batch size schedules to SGD. 
For technical reasons and to simplify the upcoming theory considerably, we require volatility controls to be continuous, which also means we allow non-integer batch sizes.
Thus, for any continuous $\al : [0,T] \to [0,1]$, we now consider a (fictitious) variant of SGD, given by 
\begin{equation}
	\label{eq:fmbSGD}
	\chi_{n+1}^{h,\al} = \chi_n^{h,\al} + h f_{nh}^{h,\al}(\chi_n^h),
\end{equation}
with
\[\E[f_t^{h,\al}(\thet)] = -\nabla \cR(\thet), \quad \Cov[f_t^{h,\al}(\thet)] = \al_t \Si(\thet),\]
for all $h\in (0,1), t\in[0,T]$ and $\thet \in \R$.
We refer to \eqref{eq:fmbSGD} as \emph{fractional batch size SGD}. 

Now, our goal is finding an \emph{optimal} sequence of batch sizes, so that the error $\E[\cR(\chi_M^h)]$ for a given final time step $M$ is minimal.

Of course, stating the problem this way suggests setting the batch size to be maximal, since this makes our estimate of the true gradient as accurate as possible. However, a higher batch size also means higher computational cost. Therefore, we will postulate the condition that the number of data points used is fixed, i.e.
\begin{equation}
\label{eq:totalBSfixed}
\sum_{n=1}^M B_n = \frac c h
\end{equation}
for some constant $c\geq T$, where we divide by $h$, with $c/h \in \N$, for convenience. For SGD \emph{without replacement}\footnote{Note that our theory technically only applies to SGD without replacement, with a \emph{single} epoch.} (which is commonly used in practice) one would usually consider $\frac c h = \text{sample size} \times \text{epochs}$. Insisting that $c\geq T$ is natural, since, for $T/h \in \N$,
\[\frac c h = \text{number of samples processed} \geq \text{number of SGD steps} = \frac T h,\]
and the lower bound is obtained by choosing batch size $1$ in each step. Suppose $B = \al^{-1}$. Then under Condition \eqref{eq:totalBSfixed},
\[c = h \sum_{n=1}^{\floor{T/h}} B_{nh} = h \sum_{n=1}^{\floor{T/h}} \frac{1}{\al_{nh}} \to \int_0^T \frac{1}{\al_t}\,dt, \quad h\downarrow 0.\]
Thus, in the continuous-time setting, condition \eqref{eq:totalBSfixed} corresponds to the following condition on the volatility control
\begin{equation}
\label{eq:subjectTo}
\int_0^T \frac{1}{\al_t} \,dt = c.
\end{equation}
Therefore, we may consider the following optimal volatility control problem: Given $c\geq T > 0$, determine
\begin{equation}
\label{eq:origDiffControlProb}
\amin_{\al\in A(L)}\E[\cR(\chi_{\floor{T/h}}^{h,\al})],
\end{equation}
where the set of admissible controls is given by
\[A(L) = \set{\al : [0,T] \to [0, 1] : \nrm{\sqrt{\al}}{\Lip} \leq L, \int_0^T \frac{1}{\al_t} \,dt = c},\]
for some sufficiently large $L > 0$. The Lipschitz condition on $\sqrt \al$ is necessary for the continuous-time theory (cf. Section \ref{sec:contTheory}) to be applicable to this problem.

Initially, one could hope to find an explicit solution to \eqref{eq:origDiffControlProb}, at least in dimension $d = 1$. However, this is very difficult or perhaps impossible. Following \cite{li_stochastic_2017}, our idea is instead to approximate the discrete-time SGD iterations using a family of continuous-time diffusion processes. Then we can apply optimal control theory to the approximating \tSDE s and solve \eqref{eq:origDiffControlProb}, \emph{up to} an error $Ch^2$, where $h$ is the learning rate and $C$ is an increasing function of the parameter $L$. The explicit solution of this relaxed problem is the content of our main result Theorem \ref{thm:optBS}. Since the goal is to find an explicit solution, a further complication arises. In most problems, the variance of the sample gradients $\Si$ is non-constant and even state-dependent. We solve this issue by expanding the diffusion approximation again into a series with respect to the learning rate. This allows for a significant simplification of the control problem.

Aside from focusing on batch size rather than learning rate schedules, our work extends the approach in \cite{li_stochastic_2017} in several aspects:

\paragraph{Summary of contributions}
\begin{itemize}
\item We establish, to our knowledge for the first time, a \emph{rigorous} theory for transferring deterministic optimal controls from a continuous-time diffusion approximation of a numerical one-step stochastic method back to discrete-time. This includes extending the theory of (second-order) stochastic modified equations in \cite{li_stochastic_2019} to allow for time-dependent drift and diffusion coefficients. Thus, we are able to study SGD with learning rate and batch size schedules in continuous-time.
\item Using perturbation theory, we reduce the continuous-time optimal control problem to a linear control problem, without resorting to unrealistic assumptions on the diffusion coefficient. In particular, in contrast to previous works, we do \emph{not} assume the variance of the sampled gradients $\Si$ to be constant and explicitly allow it to be state-dependent.
\item We demonstrate the potential of our theory by deriving an \emph{explicit} quasi-optimal batch size schedule using the continuous-time Pontryagin maximum principle.
\end{itemize}

%While introducing the parameter $L$ seems artificial, it is nevertheless important as it enters into the leading error term $C$ in \eqref{eq:weakErr2}
We remark that in practice it is reasonable to use the largest mini-batch size such that all mini-batches fit into memory. In this setting we will use the term \emph{batch size} to refer to \emph{gradient accumulation} instead, i.e.\ the number of batches until an update is made. We will no longer explicitly make this distinction, because it makes no essential difference to our theory.

\paragraph{Failure of the first-order batch size theory}
To solve a the optimal batch size control problem, at least in a relaxed sense, we expand the expected risk $\E[\cR(\chi_{\floor{T/h}}^h)]$ into a series in $h$ with a remainder term of size $h^k$ for some $k\in \N$ . Then we seek a statement of the following form. Fix $L > 0$ sufficiently large. Then there exists a $C$, depending on $L$, and a $\al^*\in A(L)$, such that
\begin{equation}
\label{eq:quasiOpt}
\left|\inf_{\al^*\in A(L)} \E \cR(\chi_{\floor{T/h}}^{h,\al}) - \E \cR(\chi_{\floor{T/h}}^{h,\al^*})\right|\leq Ch^k.
\end{equation}
For example, if we let $k = 1$, then we can approximate SGD using a continuous-time first-order approximation, e.g. (cf. \cite{li_stochastic_2017})
\[dX_t^h = -\nabla \cR(X_t^h)\,dt + \sqrt{h\al_t\Si(X_t^h)}\,dW_t.\]
The following negative result demonstrates why considering $k = 1$ in \eqref{eq:quasiOpt} is too crude for a useful theory of almost optimality of batch size schedules.
\begin{prop}
\label{prop:1stOrderFail}
Let $L > 0$. There exists a $C > 0$, such that for all $\al^* \in \cA_L$, we have
\[\left|\inf_{\al^*\in A(L)} \E \cR(\chi_{\floor{T/h}}^{h,\al}) - \E \cR(\chi_{\floor{T/h}}^{h,\al^*})\right|\leq Ch.\]
\end{prop}
\begin{proof}[Proof sketch]
Consider Theorem \ref{thm:2ndOrderSME}. A similar result shows that gradient flow
\[dX_t^0 = - \nabla \cR(X_t^0)\,dt\]
is a first-order approximation of SGD, i.e.\ there exists a $C>0$, such that
\[|\E \cR(\chi_{\floor{T/h}}^{h,\al^*}) - \E \cR(X_T^0)| \leq Ch,\]
for all $h\in (0,1)$. Moreover, this $C$ can be chosen independently of $\al \in A(L)$, and so, similarly to Corollary \ref{cor:2ndOrderSMEOptim},
\[|\inf_{\al\in A(L)} \E \cR(\chi_{\floor{T/h}}^{h,\al}) - \E \cR(X_T^0)| \leq \tilde Ch.\]
By the triangle inequality the result follows.
\end{proof}

\begin{comment}
\paragraph{Does the second-order theory succeed?}
Suppose we have two families of volatility schedules
\[(\al^t : [0,T]\to [0,1])_{t\in [0,T]}, (\be^t : [0,T]\to [0,1])_{t\in [0,T]},\]
satisfying 
\[\sup_{t\in [0,T]} \nrm{\al^t}{\Lip}, \sup_{t\in [0,T]} \nrm{\be^t}{\Lip} < \infty,\]
and
\[\E[\cR(X^{h,\al^t}_t)] = \E[\cR(X^{h,\be^t}_t)] + \cO(h^2).\]

We have
\begin{align*}
\E[\cR(X^{h,\al}_T)] - \E[\cR(X^{h,\be}_T)] = & h \left(\frac12 \cR''(X_t^0) \Delt V_t^t + \cR'(X^0_t) \Delt E_t^t\right) + \cO(h^2),
\end{align*}
where
\begin{align*}
d\Delt V_s^t = & - 2 \cR''(X_s^0)\Delt V_s^t + \Si(X_s^0)(\al_s^t - \be_s^t)\,dt, \quad \Delt V_0^t = 0.\\
d\Delt E_s^t = & - \frac12 \cR'''(X_s^0)\Delt V_s^t - \cR''(X_s^0)\Delt E_s^t\,dt,\quad \Delt E_0^t = 0.
\end{align*}
Hence,
\[\Delt V_t = \int_0^t \exp\left(-2 \int_s^t \cR''(X_r^0)\,dr\right)\Si(X_s^0)(\al_s^t - \be_s^t)\,ds\]
Assume for simplicity that $\cR''' = 0$.
Assuming $\al^t_t \neq \be^t_t$, this implies the function $1 : [0,T] \to \R, t\mapsto 1$ is a solution to the Volterra integral equation of the second kind given by
\[x_t = \int_0^t K(t,s)x_s\,ds,\quad t\in [0,T],\]
where 
\[K(t,s) = - \frac{\der_t (e^{2\be^1_{s,t}}(\al_s^t-\be_s^t)\Si(X_s^0))}{(\al_t^t-\be_t^t)\Si(X_t^0)}, \quad 0\leq s\leq t\leq T.\]
However, the bound on the Lipschitz constants and the linear growth of $\Si$ implies that $K$ is uniformly bounded. Hence, the solution of said equation is unique, even though $0$ is also a solution (contradiction).
\end{comment}
\section{Main result}
Set $d = 1$.
Given a function $g : \R \to \R$ write $g\in \Lip^l$ if $g\in C^l$ and $\der^k g$ is Lipschitz, for all $k\in \set{0,\dots, l}$.

We make the following technical assumptions on $\cR$ and $\Si$.
\begin{assum}
\label{assum:riskAndGlobEx}
The function $\cR : \R \to \R$ is in $C^5$, $\cR'\in \Lip^4$, $\sqrt \Si \in \Lip^3$ and $\Si > 0$ everywhere.
 Further, the linear growth condition
\[|\cR'(\thet)| + |\sqrt{\Si(\thet)}| \lesssim 1 + |\thet|, \thet\in \R\]
holds. Finally,
\[|\cR(\thet)| \lesssim 1 + |\thet|^2, \thet \in \R,\]
and $\cR''(X_T^0) > 0$, where $X^0$ is gradient flow (cf.\ Equation \eqref{eq:gradFlow}).
\end{assum}
Since $\cR''$ is bounded, the product $\cR'' \cR'$ is Lipschitz and of linear growth as well.

\begin{assum}
\label{assum:fNice}
There exists a random variable $Z$ with finite moments, such that
\[|f_t^{h,\al}(\thet)| \leq Z(1 + |\thet|), a.s.,\]
for all $h\in [0,1], \text{Lipschitz continuous }\al : [0,T]\to [0,1], t\in [0,T]$ and $\thet\in \R$.
\end{assum}
\begin{comment}
While the algorithm \eqref{eq:fmbSGD} remains fictitious, Assumption \assref{assum:fNice} is fairly reasonable. Indeed, recall the update rule for \emph{actual} mini-batch SGD given in \eqref{eq:mbgrads} and assume we are given a family of constants $(C_{x,y})_{(x,y)\in \R\times \R}$, such that
\[|\cR'_{x,y}| \leq C_{x,y}(1 + |\thet|),\]
and such that $C_{\bx,\by}$ has finite moments, for $(\bx,\by)\sim \nu$. Then also
\[\left|\frac{1}{B_n}\sum_{(x,y)\in \cB_n} \cR'_{x,y}(\thet)\right| \leq Z(1 + |\thet|),\quad \thet \in \R,\]
for some $Z$ with finite moments.
\end{comment}
Our main result provides an explicit relaxed solution of the optimal volatility control problem \eqref{eq:origDiffControlProb} in dimension $d = 1$.
\begin{satz}
\label{thm:optBS}
Assume \assref{assum:riskAndGlobEx} and \assref{assum:fNice} and consider fractional batch size SGD (equation \eqref{eq:fmbSGD}) with a fixed initial value $\chi_0 \in \R$.
Let $T > 0$ and consider the solution $X^0$ to the so called gradient flow ODE
\begin{equation}
\label{eq:gf1D}
d X_t^0 = - \cR'(X_t^0) \,dt, \quad X_0^0 = \chi_0,
\end{equation}
Set 
\[\be_{t,T}^1 = -\int_t^T \cR''(X^0_s)\,ds, \quad \be_{t,T}^2 = -\int_t^T \cR'''(X^0_s)\,ds,\]
\[\eta_{t,T} = \begin{cases}
	\frac{e^{-\be_{t,T}^1}-e^{-2\be_{t,T}^1}}{\be_{t,T}^1}, & \be_{t,T}^1 \neq 0,\\
	1, & \be_{t,T}^1 = 0,\end{cases}\]
\[\delt_{t,T} = e^{-2\be_{t,T}^1} \cR''(X^0_t) - \be_{t,T}^2 \eta_{t,T} \cR'(X^0_t) > 0,\]
and
\begin{equation}
\label{eq:optVolCont}
\al_t^*(\la) = \sqrt{\frac{2\la}{\delt_{t,T} \Si(X_t^{0})}} \wedge 1, \quad \la > 0,
\end{equation}
for all $t\in [0,T]$.
Then there exists a constant $\la > 0$, such that for all $L\geq \nrm{\sqrt{\al^*(\la)}}{\Lip}$, there exists constant $C > 0$, depending on $L$, with
\[|\min_{\al\in A(L)} \E\cR(\chi_{\floor{T/h}}^{h,\al}) - \E\cR(\chi_{\floor{T/h}}^{h,\al^*})| \leq Ch^2, \quad h\in (0,1).\]
%Further, $\al^*$ is the optimum in \eqref{eq:diffControProbO32} under the dynamics
%\begin{equation}
%\label{eq:SME2fmb}
%dX_t^h = - \cR'(X_t^h) - \frac h2 \cR''(X_t^h)\cR'(X_t^h)\,dt + \sqrt{h \al_t %\Si(X_t^h)}\,dW_t,
%\end{equation}
%with $X_0^h = \chi_0$, and series expansion \eqref{eq:Xexpans}.
\end{satz}
Here $\wedge = \min$. The proof of Theorem \ref{thm:optBS} is postponed to Appendix \ref{sec:proofMain}.
 
\section{Continuous-time theory of mini-batch SGD}
\label{sec:contTheory}
The proof of Theorem \ref{thm:optBS} relies crucially on a continuous-time theory of SGD and results for relating discrete and continuous time.
There are three main steps to proving our main result: 
\begin{enumerate}[(i)]
	\item approximating SGD with a family of \tSDE s indexed by the learning rate,
	\item applying perturbation theory to the approximating family of \tSDE s, thereby expanding it again into a series with respect to the learning rate,
	\item stating and solving an optimal control problem for this series expansion.
\end{enumerate}
Finally, we transfer the solution to the latter optimal control problem back to the discrete SGD process. In this section we briefly sketch these ideas while details are referred to the Appendices.

\subsection{Diffusion approximation}
Denote by $\nabla^2 f$ the Hessian matrix of a function $f\in C^2(\R^d)$. 
Set $b^0 := - \nabla \cR$ and $b^1 := - \frac 14 \nabla | \nabla \cR|^2$. 
Roughly following Li et. al \cite{li_stochastic_2019}, the dynamics of \eqref{eq:fmbSGD} can be approximated by the $h$-indexed family of \tSDE s 
\begin{equation}
	\label{eq:mbSME}
	dX_t^h = b^0(X_t^h) + h b^1(X_t^h) \,dt + \sqrt{h \al_t \Si(X_t^h)}\,dW_t,
\end{equation}
We also denote the solution of \eqref{eq:mbSME} for a given volatility control $\al$ and $h\in (0,1)$ by $X^{h,\al}$.

We refer to Equation \eqref{eq:mbSME} as a \emph{weak second-order diffusion approximation} of \eqref{eq:mbSGD}, since, under reasonable conditions, for all $T > 0$ there exists a $C > 0$, such that for all smooth $g : \R^d \to \R$ with derivatives of at most polynomial growth, we have
\begin{equation}
	\label{eq:weakErr2}
	\max_{n\in \set{0,\dots, \floor{T/h}}} |\E[g(\chi_n^h)] - \E[g(X_{nh}^h)]| \leq Ch^2,
\end{equation}
for all $h\in (0,1)$, given that the diffusion approximation and SGD have the same starting point, that is $X_0^h = \chi_0$.

In contrast to this diffusion approximation, in the literature on SGD one commonly considers the gradient flow ODE 
\begin{equation}
	\label{eq:gradFlow}
	dX_t^0 = -\nabla \cR(X_t^0)\,dt
\end{equation}
as a continuous-time version of SGD. This is not sufficient for an analysis of batch sizes, since the dynamics only depend on the mean of the sampled gradients. On the other hand, batch sizes only appear in the covariance matrix of the gradient noise, which is why we consider the stochastic dynamics \eqref{eq:mbSME} instead. Putting that aside, the approximation quality of \eqref{eq:gradFlow} is worse compared to \eqref{eq:mbSME} since it is merely of first-order, i.e.\ for all $T > 0$ there exists a $C > 0$, such that for all smooth $g : \R^d \to \R$ with derivatives of at most polynomial growth, we have
\begin{equation}
	\label{eq:weakErrGF}
	\max_{n\in \set{0,\dots, \floor{T/h}}} |\E[g(\chi_n^h)] - \E[g(X_{nh}^0)]| \leq Ch,
\end{equation}
for all $h\in (0,1)$, given that $X_0^0 = \chi_0$.

Under reasonable conditions we can make the constant $C$ in \eqref{eq:weakErr2} independent on the choice of volatility control. This allows us, in some sense, to replace the discrete time control problem \ref{eq:origDiffControlProb} with the continuous-time control problem
\begin{equation}
	\label{eq:contTimeDiffControlProb}
	\amin_{\al\in A(L)}\E[\cR(X_T^{h,\al})],
\end{equation}
so that we can use tools from stochastic calculus and continuous-time optimal control. Details for this transfer from discrete to continuous time are deferred to Appendix \ref{sec:proofSME2}.

Unfortunately, Problem \eqref{eq:contTimeDiffControlProb} is still too difficult to be solved explicitly, primarily because of the covariance matrix $\Si$. For example, even in one-dimensional linear regression tasks $\Si$ is already a quadratic polynomial and there is generally no hope that $\Si$ simplifies, say, to a constant.

To rectify this issue, in the subsection we introduce an expansion of \eqref{eq:mbSME} with respect to the learning rate.

\subsection{Expansions in the learning rate}
Consider again the approximation result \eqref{eq:weakErr2}. Based on this we can approximate the risks
\[|\E \cR(\chi_n^h) - \E \cR(X_{nh}^h)| = \cO(h^2).\]
However, if we expand the risk of the diffusion approximation into a Taylor series with respect to the learning rate $h$ as follows
\[\cR(X_t^h) = \cR^{(0)}_t + h \cR^{(1)}_t + \cO(h^2),\]
then all terms beyond $h^2$ are not known to contribute (positively or negatively) to the approximation error in \eqref{eq:weakErr2}.
Therefore, in order to find optimal batch sizes for \eqref{eq:mbSGD} we do not lose any accuracy if we change \eqref{eq:contTimeDiffControlProb} such that we minimize
\[\E[\cR^{(0)}_t + h \cR^{(1)}_t]\]
instead.

We can find $R^{(q)}$ by also considering a series expansion for the diffusion approximation
\begin{equation}
	\label{eq:Xexpans}
	X_t^h = X^0_t +  \sqrt h X^{(1/2)}_t + h X^{(1)}_t + h^{3/2} X^{(3/2)}_t + \cO(h^2)
\end{equation}
Then one can derive a system of \tSDE s for \\$X^0, X^{(1/2)},\dots$ which is in a triangular form and such that the equations for $X^{(1/2)}, X^{(1)}$ and $X^{(3/2)}$ are linear, given $X^0$.

Given the expansion \eqref{eq:Xexpans}, one can show that for $\cR \in C^2(\R)$ we have
\begin{equation}
	\label{eq:riskExp}
	\E[\cR(X^h)] = \cR(X^0) + h \left(\frac12 \cR''(X^0)\Var[X^{(1/2)}] + \cR'(X^0)\E[X^{(1)}]\right) \\
	+ \cO(h^2),
\end{equation}
conditional on the initial condition $X_0 = \chi_0$. Here, $X^0$ is gradient flow, as in equation \eqref{eq:gradFlow}. Note that the process $X^{(3/2)}$ introduced in \eqref{eq:Xexpans} plays no role in the expansion of the expected risk. Further, we have
\begin{align}
	d\Var[X^{(1/2)}_t] = & 2 \cR''(X^0_t) \Var[X^{(1/2)}_t] + \al_t \Si(X_t^0) \,dt, \label{eq:varx12R}\\
	d\E[X_t^{(1)}] = & \frac12  \cR'''(X^0_t) \Var[X^{(1/2)}_t] +  \cR''(X^0_t) \E[X_t^{(1)}] + b^1(X_t^0) \,dt \label{eq:ex1R},
\end{align}
In essence, in \eqref{eq:riskExp}, we are correcting the mean risk of gradient flow by terms depending on the learning rate $h$, the randomness inherent to SGD and the fact that even deterministic gradient descent with finite learning rate essentially optimizes the modified objective
\[\cR + \frac h4 | \nabla \cR|^2,\]
which is evident from the drift coefficient in equation \eqref{eq:mbSME}.

Since gradient flow does not depend on the volatility control, our problem simplifies to
\begin{equation}
	\label{eq:diffControProbO32}
	\amin_{\al\in A(L)} \frac12 \cR''(X^0_T)\Var[X^{(1/2), \al}_T] + \cR'(X^0_T)\E[X^{(1), \al}_T],
\end{equation}
where we indicated the dependence of $X^{(1/2)}$ and $X^{(1)}$ on $\al$.

\subsection{Batch size control}
In order to solve \eqref{eq:diffControProbO32} we take a look at the Lagrange dual problem, i.e.\ for $\la > 0$ we consider 
\begin{equation}
\label{eq:diffControProbO2Risk}
	\amin_{\al\in A'(L)}  \frac12 \cR''(X_T^0) \Var[X^{(1/2), \al}_T] + \cR'(X_T^0)\E[X_T^{(1), \al}] + \la \int_0^T \frac{1}{\al_t}\,dt,
\end{equation}
where $\Var[X^{(1/2),\al}]$ and $\E[X^{(1),\al}]$ satisfy \eqref{eq:varx12R} and \eqref{eq:ex1R}, respectively, and 
\[A'(L) = \set{\al : [0,T]\to [0,1] : \nrm{\sqrt{\al}}{\Lip} \leq L}.\]
If $\al^*(\la)$ is a solution to \eqref{eq:diffControProbO2Risk} and there exists a $\la > 0$ with 
\begin{equation}
\label{eq:lambdaC}
\int_0^T \frac{1}{\al_t^*(\la)}\,dt = c,
\end{equation}
then $\al^*(\la)$ solves the primal problem \eqref{eq:diffControProbO32}. 

To solve \eqref{eq:diffControProbO2Risk}, we apply the Pontryagin maximum principle (cf.\ \cite{pham2009continuous} Chapter 6.4 for more details on the maximum principle) to the two-dimensional system of linear equations, \eqref{eq:varx12R} and \eqref{eq:ex1R}. This is relatively straightforward and yields the optimal volatility control \eqref{eq:optVolCont}. Details can be found Appendix \ref{sec:optCont}.

\section{Optimal batch sizes for linear regression}
In this section we apply Theorem \ref{thm:optBS} to the problem of linear regression with mini-batch SGD.

\subsection{The statistical learning setting}
\newcommand{\bep}{\bm{\varepsilon}}
Suppose we are given random variables $\bx$ and $\bep$ defined on a probability space $(\Om, \cF, \P)$, such that $\bx$ and $\bep$ are independent, $\E \bep = 0, \si_\ep^2 := \E \bep^2 < \infty$ and $\E\bx^4 < \infty$.
Let $\thet^* \in \R^d$. We define the $\R$-valued random variable $\by$ by
\[\by = \thet^* \bx + \bep.\]
Denote the distribution of $(\bx, \by)$ by $\nu$. We call $\nu$ the \emph{population}. We consider applying SGD to a sequence of \tiid data points $(\bx_0, \by_0), (\bx_1, \by_2),\dots$, drawn from $\nu$, which follows a linear model.
The population is considered unknown to us.

Let $\ell$ be the \emph{square loss}, given by $\ell(y,y') = \frac12(y-y')^2$. The goal is to fit the data drawn from $\nu$ using a linear predictor $\thet \mapsto  \thet x$.
Thus, for any data point $(x,y) \in \R\times \R$ we consider the squared risk
\[R_{x,y}(\thet) = \ell(\thet x, y) = \frac12 (\thet x - y)^2.\]
We define the \emph{population risk} by
\[\cR(\thet) := \E[R_{\bx,\by}(\thet)].\]
We stress that the bold letters $\bx, \by$ denote random variables, while $x,y$ represent realizations.
The minimum of $\cR$,  i.e.\ the best possible fit, is given by the population parameter $\thet^*$.

Then, we have
\begin{gather*}
	\cR(\thet) 	= \frac12 \ka(\thet-\thet^*)^2+ \cR^*, \quad
	\cR'(\thet) 	=  \ka(\thet - \thet^*), \quad
	\cR''(\thet) =  \ka,
\end{gather*}
where $\ka := \Var \bx$ and $\cR^* := \inf_{\theta\in \R} \cR(\thet) = \frac{\si_\ep^2}{2}$ is the smallest possible population risk. Further,
\[\Si(\thet) =  \Var[\der_\thet \ell(\thet \bx,\by)] = \ka^2(\Kurt \bx - 1)(\thet- \thet^*)^2 + 2 \ka \cR^*,\]
where $\Kurt(\bx) := \E[\bx^4]/\ka^2$ is the \emph{kurtosis} of $\bx$. Note that, e.g., $\Kurt \bx = 3$ if $\bx \sim \cN(0,\ka)$.

\subsection{Optimal volatility}
Consider Theorem \ref{thm:optBS}, now in the case of linear regression as outlined in the previous subsection.
Gradient flow satisfies
\[dX^0_t = - \ka (X^0_t - \thet^*)\,dt, X_0^0 = \chi_0.\]
and so \[X^0_t = \left(\chi_0 - \thet^*\right)e^{-\ka t} + \thet^*.\]
Define the \emph{excess population risk} $\cR^e = \cR - \cR^*$ and the \emph{initial excess population risk} $\cR^e_0 = \cR(\chi_0) - \cR^*$.
Then the excess population risk of gradient flow at time $t$ satisfies $\cR^e(X_t^0) = \cR_0^e e^{-2\ka t}$.
Thus,
\[\Si(X_t^0) = 2\ka((\Kurt \bx - 1) \cR_0^e e^{-2\ka t} + \cR^*).\]
Coming back to the solution of the control problem given by Theorem \ref{thm:optBS}, we have $\be_{t,T}^1 = - \ka (T-t)$ and $\be_{t,T}^2 = 0$. Hence, 
\[\delt_{t,T} = e^{-2\be_{t,T}^1} \cR''(X_T^0) = \ka  e^{-2\ka(T-t)},\]
and the optimal volatility control is
\begin{align}
	t \mapsto &\sqrt{\frac{\la}{\ka^2 e^{-2\ka(T-t)}((\Kurt \bx - 1) \cR_0^e e^{-2\ka t} + \cR^*)}} \wedge 1. \nonumber
\end{align}
After a linear re-parameterization and setting $\ga := \frac{R_0^e}{R^*}(\Kurt \bx - 1)$, we have
\begin{equation}
\label{eq:optVolLinReg}
\al_t^*(\la) = \sqrt{\frac{\la}{\ga + e^{2\ka t}}} \wedge 1, \quad t\in [0,T], \la > 0.
\end{equation}
For $\la > 0$ such that \eqref{eq:lambdaC} is satisfied, $\al^*(\la)$ is the optimal volatility control for the linear regression problem. Figure \ref{fig:optvolschedla300ga280} shows $\al^*$ for different values of $\la$ and $\ga$.
In the case that the upper bound of $1$ is never attained, $\la$ can be calculated explicitly (cf.\ Appendix \ref{sec:propOfOptVolLinReg}).
\begin{figure}
	\centering
	\fbox{\includegraphics[width=0.45\linewidth]{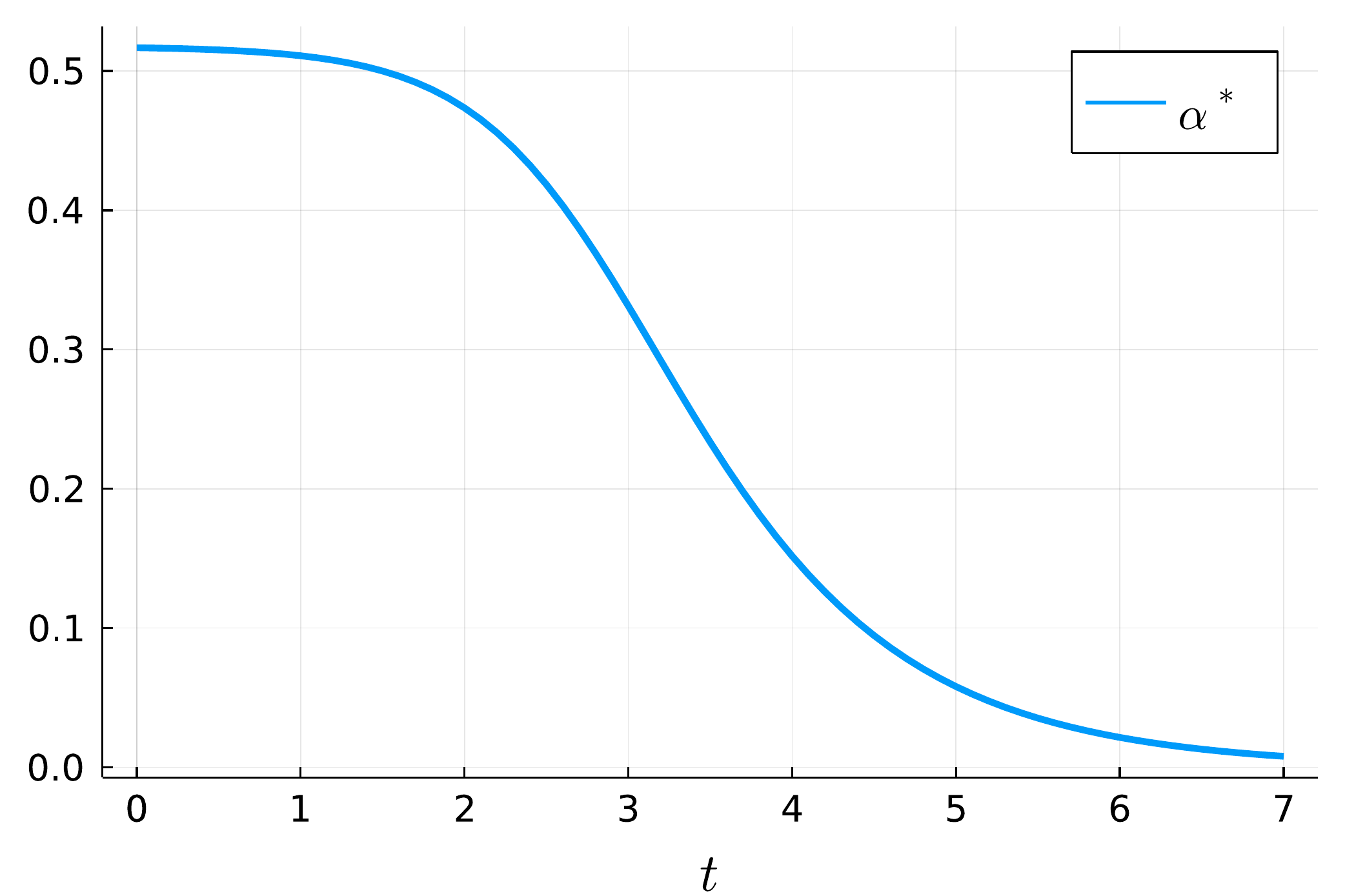}\hspace*{0.05\linewidth}\includegraphics[width=0.45\linewidth]{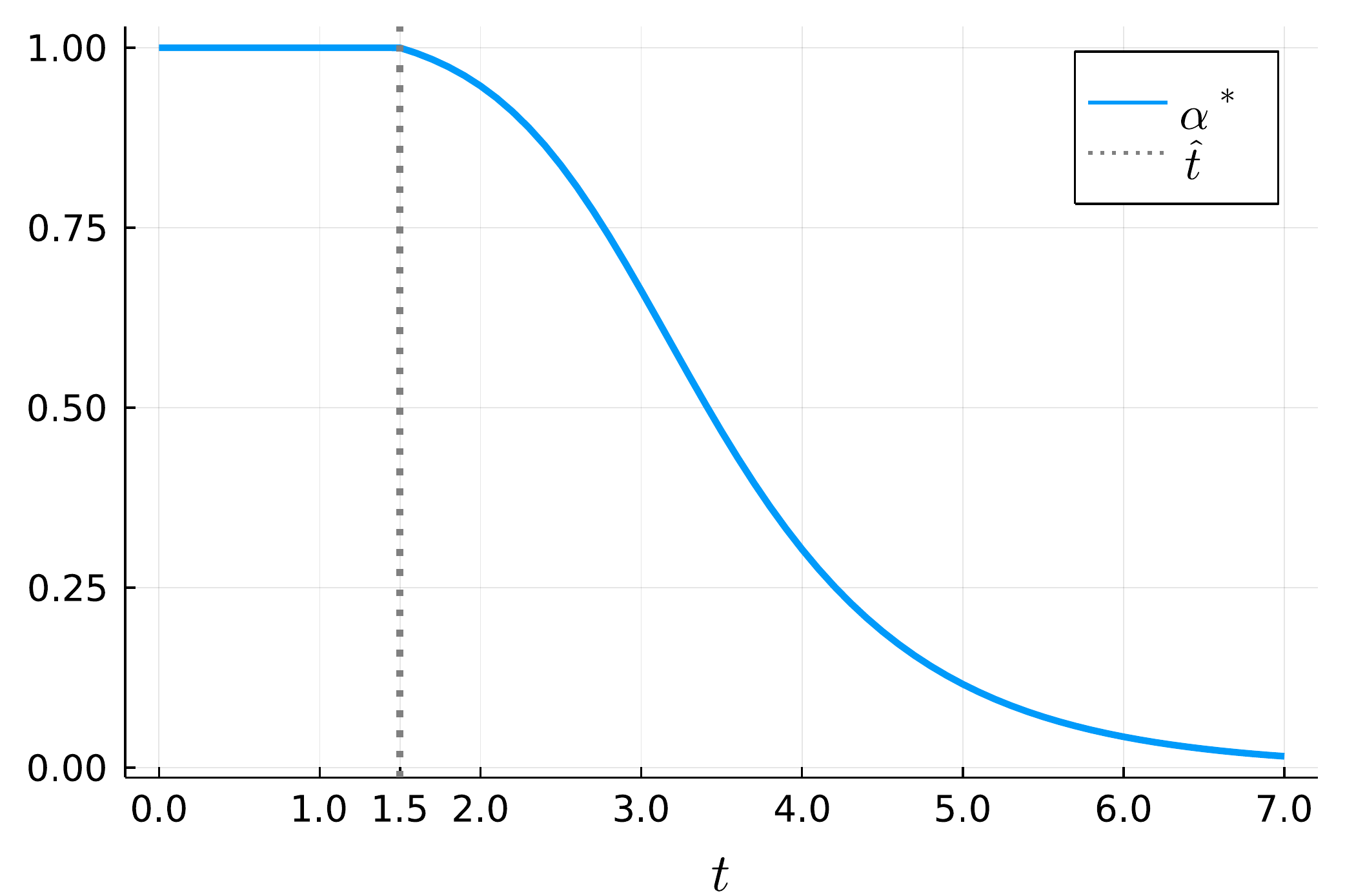}}
	\caption{Optimal volatility control $\al^*$ for linear regression, with $\la = 75$ (left) / $300$ (right), $\ga = 280$ and $\ka = 1$. On the right, the time point $\hat t \approx 1.5$, where volatility switches away from $1$ is indicated by the dotted, vertical line. On the left, we have $\al^* < 1$ everywhere.}
	\label{fig:optvolschedla300ga280}
\end{figure}
Note that the optimal volatility control $\al^*$ in \eqref{eq:optVolLinReg} is non-increasing. 
Hence, for every $\la > 0$ there exists a unique $\check t(\la) \in [0,T]$ with $\al_t^*(\la) < 1$ for all $t \in [\check t, T]$. In fact, we have 
\[\frac{\la}{\ga + e^{2\ka t}} = 1 \Leftrightarrow t = \frac{1}{2\ka} \ln(\la - \ga),\]
provided $\la - \ga \geq 1$.
Hence, the time point where we switch away from volatility $1$ is given by
\[\check t(\la) = \begin{cases}
	\frac{1}{2\ka} \ln( \la - \ga), & \la > \ga + 1, \\
	0, & \text{else}.
\end{cases}\]
\begin{comment}
Write $z := \sqrt{\la a}$.
We calculate,
\begin{align*}
\int_0^T \frac{1}{\al_t^*(\la)}\,dt = & \hat t(\la) + \frac1 z\int_{t(\la)}^T  C(t),dt \\
= & \hat t(\la) + \frac 1 {\ka z}(F(T) - F(\hat t(\la))),
\end{align*}
where $C(t) = \sqrt{b + e^{2\ka t}}$ and 
\[F(t) = C(t) + \sqrt b \operatorname{ArcTanh}\left(\frac{C(t)}{\sqrt b}\right).\]
\end{comment}

\subsection{A numerical example}
In this subsection we use the optimal volatility control \eqref{eq:optVolLinReg} for numerically estimating the true parameter $\thet^*$ in a linear regression problem, using mini batch SGD. Experimental details are deferred to Appendix \ref{sec:experiment}.

Figures \ref{fig:batches256-lr0} depict the results of two runs of the experiment for different parameter values.
As expected, increasing the batch size leads to lower population risk at the end of training. In the examples, the difference to using constant batch size can be more than one order of magnitude. Also, in Figure \ref{fig:batches256-lr0} we see that, additionally, in the early stages of training, we can use lower batch sizes than the constant schedule for significantly faster convergence, in terms of samples processed.
It should be pointed out that the effects of optimized batch schedules are more prominent fo
longer training times, since there is a greater range of batch sizes one can use. Conversely, if we have too few iterations, then the \enquote{optimal} and constant schedules coincide.

\begin{figure}
	\centering
	\fbox{\includegraphics[width=0.45\linewidth]{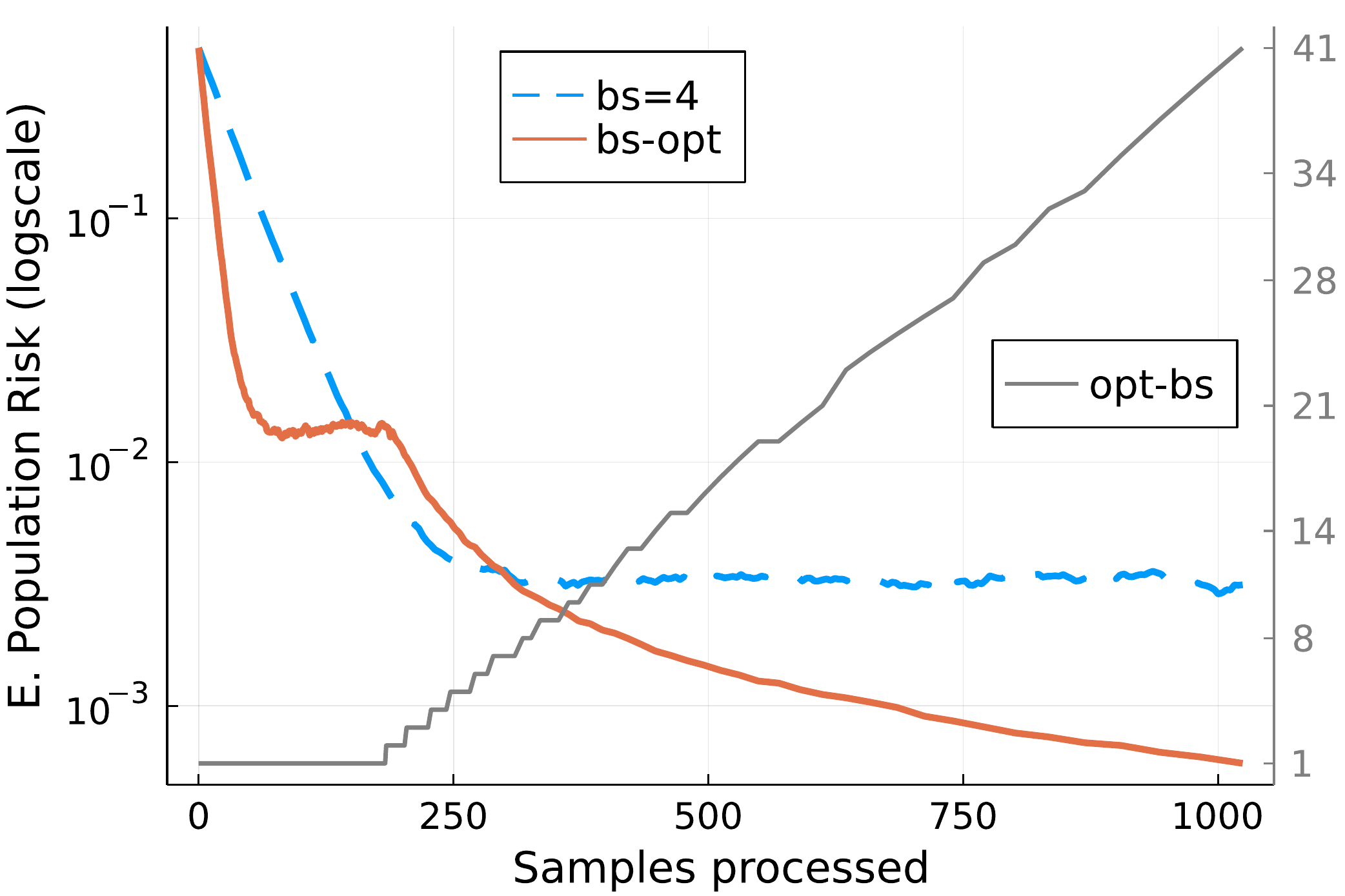}
		\hspace*{0.05\linewidth}
	\includegraphics[width=0.45\linewidth]{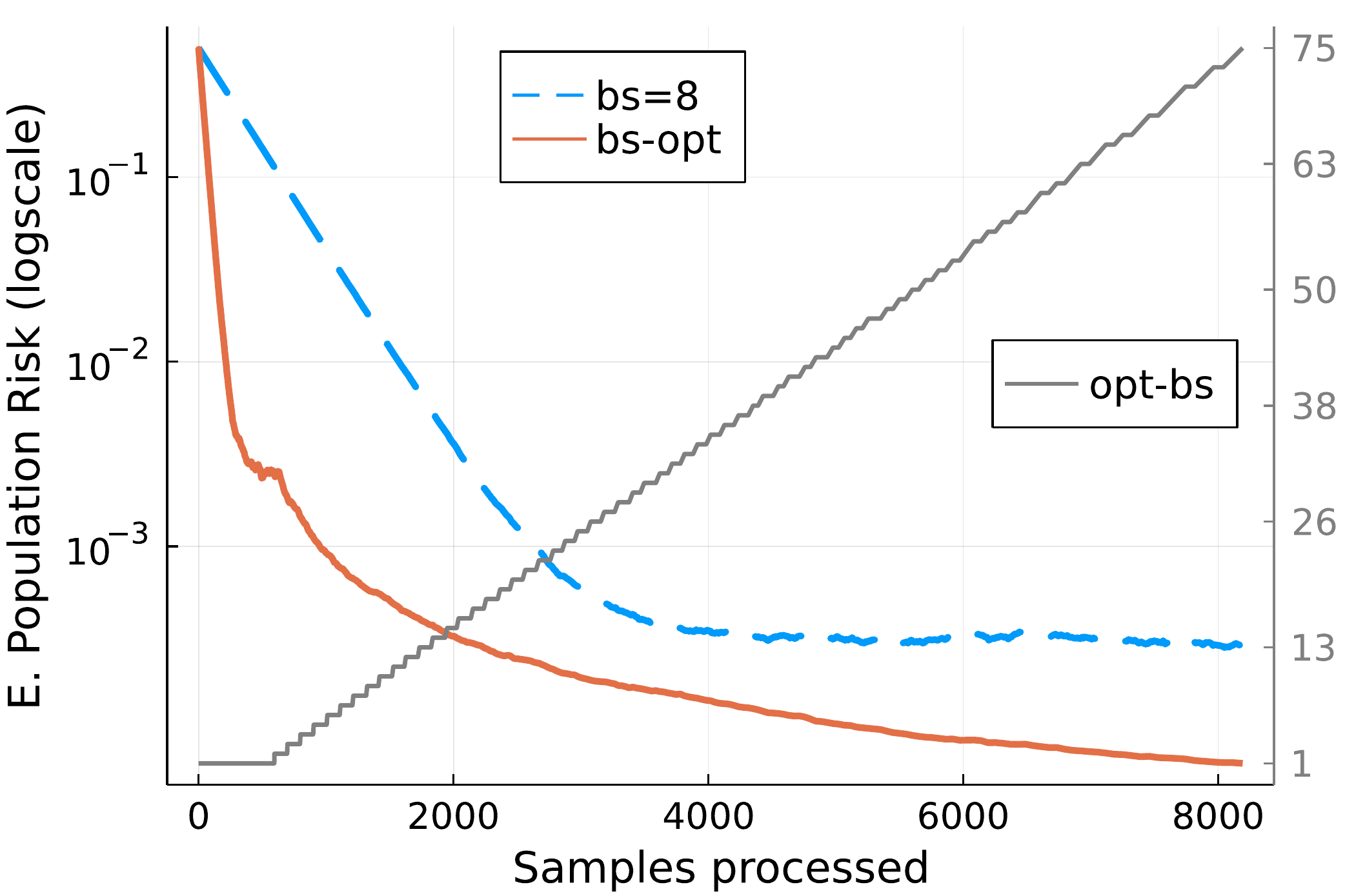}}
	\caption{Excess population risk of mini batch SGD as a function of the number of samples processed, averaged over $1000$ instances each, with constant batch size of $4$ (left) / $8$ (right) and using an \enquote{optimal} batch size schedule (bs-opt). Here, the sample size is $N = 2^{10}$ / $2^{13}$, the number of steps, i.e.\ batches, is $M = 256$ / $2^{10}$, and the learning rate is $h = 0.05$ / $0.01$. The right $y$-axis specifies the number of samples used by bs-opt.}
	\label{fig:batches256-lr0}
\end{figure}

\section{Limitations}
There are several limitations to the main result Theorem \ref{thm:optBS}. 

Firstly, the dimension is fixed to $1$. However, we suspect that the behavior of our quasi-optimal volatility control also yield great benefits in higher dimensions. A large portion of the theory could in principle be developed in higher dimensions. Unfortunately, in this case the optimal control problem \eqref{eq:contTimeDiffControlProb} cannot be reduced to a problem of controlling a system of \tODE s. Instead, one needs to consider systems of non-linear fully-coupled forward backward \tSDE s and resort to numerical methods for computing the optimal control. Solving high-dimensional non-linear FBSDEs is again a difficult problem, which requires using deep learning techniques (cf. \cite{ji_three_2020}). That makes it prohibitively expensive to use such a method in practice. Alternatively, one could study a continuous-time mean-field approximation of SGD applied to high-dimensional problems (cf.\ \cite{gess_stochastic_2023}).

Secondly, the optimal volatility control depends on the gradient flow solution, which generally cannot be derived explicitly. Moreover, it would be more natural for the optimal control to be Markov, i.e.\ a function of the current parameter iterate $\chi_n$. However, this would require developing sophisticated approximation results, based on causal optimal transport, that allow for the transfer of stochastic optimal controls (cf. \cite{acciaio_causal_2017}).

Thirdly, our results only apply to the fictitious fractional batch size SGD. If we round our optimal schedule in any way, then the optimality result would only hold up to a term of order $1$ in the learning rate, which is crude. Extending diffusion approximations to allow for discontinuous volatility controls is a difficult issue and would likely change the approximating equations to feature local times, since we have to resort to using the \tIto-Tanaka formula when deriving the stochastic Taylor approximations (cf.\ Proposition \ref{prop:itoTaylor2} in Appendix \ref{sec:proofSME2}).

Fourthly, several quantities featured in the optimal volatility schedule are difficult to compute or estimate in practice. This includes the integrals $\int_t^T$ looking forward in time, the Lagrange mutliplier $\lambda$, as well as population parameters, such as $\cR^*$ (cf. Equation \ref{eq:optVolLinReg}).

Finally, the assumptions on the coefficient of the diffusion approximation \assref{assum:riskAndGlobEx} are technical, restrictive and sometimes violated in examples. Lipschitz and linear growth conditions are standard in the \tSDE{} literature to ensure existence and uniqueness of global solutions, but they can be significantly relaxed, possibly even to the point of considering weak solutions. The smoothness and boundedness of the derivatives of the coefficients is used mainly to derive a result on differentiation with respect to the initial condition.

\section{Related Work}
\paragraph{Batch size schedules}
In practice it is common to chose a constant batch size. However, it has been observed before that increasing batch size during training of neural networks can be beneficial (cf.\ \cite{smith_dont_2018}, \cite{friedlander_hybrid_2012}, \cite{byrd_sample_2012}, \cite{balles_coupling_2017}, \cite{bottou_optimization_2018}, \cite{de_automated_2017}). The batch size schedules derived in these works are based on useful heuristics. In contrast, we use optimal control theory for deriving a theoretically (quasi-) optimal schedule. While some of these works emphasize an equivalence of increasing batch size and decreasing learning rate, our theory breaks this symmetry, by using the (maximal) learning rate $h$ for development of the continuous-time approximation. Further, we remark that learning rate \emph{schedules} affect the dynamics of gradient flow, while the batch size, which only affects volatility, does not.

Finally, the idea of deriving optimal batch size schedules using diffusion approximations was also studied by Zhao et.\ al in \cite{zhao2022batch}, which we were unaware of at the time of writing this paper. One of the great the strengths of their paper is that they derive their schedule in higher dimensions, which increases its applicability significantly compared to our work.
However, we still feel our article has several theoretical strengths over \cite{zhao2022batch}:
\begin{enumerate}[(a)]
\item They assume throughout that their objective function is quadratic. This is e.g.\ the case for linear regression, which we also study in Section 4. However, our main Theorem \ref{thm:optBS} makes no such assumption and holds for quite general objective functions.
\item Equation (3) in \cite{zhao2022batch} is a first-order approximation of SGD and therefore of worse quality (i.e\ in having a non-zero linear error term) than the second-order approximation(s) we use. Specifically, it is a good approximation only for much smaller learning rates compared to the approximation we consider (because if, say, $h = 10^{-3}$, then already $h^2 = 10^{-6}$). In fact, gradient flow is also a first-order approximation of SGD which does not contain the batch size at all, but is still not known to be worse than (3). Therefore, up to an error of $h$, any batch size schedule (barring Lipschitz assumptions, etc.) is \enquote{optimal} for SGD. This is the content of Proposition \ref{prop:1stOrderFail}.
\item In Section 4 of \cite{zhao2022batch} $\Sigma$ is assumed to be constant. We went to great lengths to avoid this commonly made assumption, because it would reduce the quality of our approximation from second to first-order.
Instead we deal with state-dependent diffusion coefficients using the perturbation theory approach, retaining the second-order approximation quality.
\item Theorem 4.2 in \cite{zhao2022batch} gives the optimal control for the SDE approximation, but does not say anything directly about SGD. In contrast, our main theorem pertains directly to (fractional batch size) SGD (see last inequality in Theorem \ref{thm:optBS}).
\end{enumerate}
In the future it would be interesting to see whether the methods of \cite{zhao2022batch} and our work can be combined to derive even better results.

\paragraph{Diffusion approximations}
Continuous-time diffusion approximations to SGD, also known as \emph{stochastic modified equations}, have been heuristically introduced in \cite{mandt_continuous-time_2015} and \cite{li_stochastic_2017}, and theoretically substantiated in \cite{li_stochastic_2019}. Since then numerous works have used diffusion approximations to study SGD (\cite{ali_implicit_2020}, \cite{an_stochastic_2020}, \cite{boffi_continuous-time_2020}, \cite{xie_diffusion_2021}, \cite{pesme_implicit_2021}, \cite{gu_sde_2021}, and others).
Further, \cite{li_stochastic_2017} was also the first work, to our knowledge, to use optimal control theory for hyperparameter tuning of SGD, by deriving an optimal learning rate control for a first-order diffusion approximation with constant diffusion coefficient. While we focus on batch size control, our work extends \cite{li_stochastic_2017} in several aspects: we establish a rigorous theory for transferring optimal controls from continuous-time theory back to discrete-time theory; we use the more accurate second-order diffusion approximation; we specifically allow for state-dependent diffusion coefficients. Further, we extend the theory in \cite{li_stochastic_2019} to allow generally for time-dependent drift and diffusion coefficients, e.g.\ learning rate and batch size schedules.

\section{Conclusion}
We have developed a continuous-time theory for calculating quasi-optimal hyperparameter schedules for stochastic gradient descent and similar stochastic one-step optimization methods, and demonstrated its usefulness by deriving a quasi-optimal batch size schedule for SGD and a large class of regression problems. Generalizing these results to allow for Markov controls, higher dimensions and more general assumptions on the drift and diffusion coefficients of the diffusion approximations, as well as the development of practically relevant algorithms, is left to future work.

%\end{document}

\begin{alphasection}
\section{Preliminaries}	
In this section we introduce notation for the upcoming appendices, as well as some basic properties. 

We write $\N = \set{1,2,\dots}$ and $\N_0 = \set{0,1,\dots}$.
A (unordered) \emph{multi-index} $\al$ is a multi-subset of $\set{1,\dots, d}$, i.e.\ a function $\al : \set{1,\dots, d} \to \N_0$. The size $|\al|$ of $\al$ is given by
\[|\al| := \sum_{j=1}^d \al(j).\]
Every subset $A\subseteq \set{1,\dots, d}$ becomes a multi-set by identifying it with its indicator function. 
Given multi-indices $\al$ and $\be$ we write $\al \leq \be$ if $\al(j)\leq \be(j)$ for all $j\in \set{1,\dots, d}$ and in that case the multi-index $\be - \al$ is well defined, by component-wise subtraction. Further, write $j\in \al$ if $\set{j} \leq \al$ and set $\al - j := \al - \set{j}$ in that case.

If a function $f : \R^d \to \R$ is $l$-times continuously differentiable, then by Schwarz's theorem the partial derivative with respect to a multi-index $\al$ with $|\al| \leq l$ is well-defined recursively, by
\[\der^\al f = \der_j \der_{\al - j} f, \der_\varnothing f = f.\]
where $j$ is any $j\in \set{1,\dots, d}$ with $j\in \al$.
Given $x\in \R^d$ and a multi-index $\al$ we define
\[x^\al := \prod_{j=1}^d x_j^{\al(j)}.\]
We denote by $A^\dagger$ the transpose of a matrix $A$.

Fix $T > 0$, $d\in \N$ and let $B\in \set{\R^d,\R^{d\times d}}$.
Consider a function $g : D\to B$, where $D$ is a subset of Euclidean space, typically $D\in \set{[0,T], \R^d, [0,T]\times \R^d}$.

We write $g\in C^l(D)$ if the function $g$ is $l$-times continuously differentiable on the interior of $D$ and it and its derivatives up to order $l$ admit a continuous extension to $D$.

Define
\begin{align*}
\nrm{g}{G_\ka} := &\sup_{x\in D} \frac{|g(x)|}{1 + |x|^\ka}, \quad \ka \in \N_0, \quad \nrm{g}{\Lip} := \sup_{\stackrel{x,y\in D}{x\neq y}} \frac{|g(x) - g(y)|}{|x-y|}.
\end{align*}
Further, for $g\in C^l(D)$ we set
\begin{align*}
\nrm{g}{G^l_\ka} := &\max_{|\al|\leq l} \nrm{\der^\al g}{G_\ka},\quad \ka \in \N_0, \quad \nrm{g}{\Lip^l} := \max_{|\al|\leq l} \nrm{\der^\al g}{\Lip},
\end{align*}
where the maximum is taken over all multi-indices $\al : \set{1,\dots, d} \to \N$ with $|\al| = l$. Moreover, given that \[\ka = \inf\set{\ka \in \N_0 : \nrm{\der^\al g}{G_\ka} < \infty, |\al|\leq l} \in \N_0 \cup \set{\infty},\]
we set $\nrm{g}{G^l} := \nrm{g}{G_\ka^l}$. Note that $\nrm{\blnk}{E^l}$ is a norm on the vector space $E^l = \set{g\in C^l(D) : \nrm{g}{E^l} < \infty}$, for $E \in \set{G_\ka, G, \Lip}$. We write $G := G^0$. 

Now, consider specifically a function $g : [0,T]\times \R^d\to B$, depending on time and space. In this context, we denote time derivatives by $\der_t$, and iterated space derivatives by $\der^\al$, for any multi-index $\al$. We write $g\in C^{k,l}([0,T]\times \R^d)$ if $g$ is $k$-times partially differentiable on $(0,T)$ in time, and $l$-times in space, and $\der_t^m \der^\al g$ has a continuous extension to $[0,T]\times \R^d$, for all $m \leq k$ and $|\al|\leq l$. Further, we write $g\in G^{k,l}([0,T]\times \R^d)$ if $g\in C^{k,l}([0,T]\times \R^d)$ and $\der_t^m \der^\al g \in G([0,T]\times \R^d)$, for all $m \leq k$ and $|\al|\leq l$.
Also,we define 
\[\nrm{g}{\Lip^\tme} : \R^d\to [0,\infty], x\mapsto \nrm{g(x)}{\Lip}.\]
This special notation is created so that we may write $\nrm{g}{\Lip^\tme} \in G(\R^d)$.

Finally, if $I$ is a set and we are given $g : I\times D\to B$ with $g_i \in C^l(D)$ for all $i\in I$, then we write 
\[g_i\in E^l, \text{uniformly in } i\in I,\]
if $\sup_{i\in I} \nrm{g_i}{E^l} <\infty$, for $E \in \set{G_\ka, G, \Lip}$.

Now, let $X = (X_t)_{t\geq 0}$ be a continuous-time stochastic process. Given $p\in [1,\infty)$ we define
\[\nrm{X}{\Lip,p} = \sup_{0\leq s\leq t\leq T} \frac{\nrm{X_t - X_s}{p}}{t-s},\]
provided it exists. Similar to before, we also define $\nrm{X}{\Lip_p^\tme}$ if $X$ depends on $x\in \R^d$ as well.
Consider random fields $X,Y : \Om \times [0,T]\times \R^d \to \R^d$ with $\nrm{X}{\Lip_p^\tme}, \nrm{Y}{\Lip_p^\tme} \in G(\R^d)$. Then also $\nrm{X+Y}{\Lip_p^\tme} \in G(\R^d)$. Further, $\nrm{X_0}{p}, \nrm{Y_0}{p} \in G(\R^d)$ implies $\nrm{X_t}{p}, \nrm{Y_t}{p} \in G(\R^d)$, uniformly in $t$, and then $\nrm{XY}{\Lip^\tme_p} \in G(\R^d)$. Similar statements apply to functions $f,g : [0,T]\times \R^d\to \R^d$.
Given $p\in [1,\infty)$ and $t\geq 0$, we further define
\[\label{def:nrms}\nrm{X}{p,t} = \left(\E\int_0^t |X_s|^p\,ds\right)^{1/p}, \quad \nrmps{X}{p,t} = \left(\E \sup_{s\in [0,t]} |X_s|^p\right)^{1/p}\]
If we are given $(X_t)_{t\in [0,T]}$, then we also write $\nrm{X}{p} := \nrm{X}{p,T}$ and $\nrmps{X}{p} := \nrmps{X}{p,T}$.
%Although usually $X$ will be $\R^d$-valued and then $|\blnk|$ refers to the Euclidean norm, these definitions naturally extend to $\R^{d_1\times \dots \times d_r}$-valued processes as well.
Similarly, given discrete-time stochastic process $\chi$ we define
\[\nrmps{\chi}{p,n} = \left(\E \max_{n'\in \set{0,\dots, n}} |\chi_{n'}|^p\right)^{1/p}.\]
In the following we will frequently omit the domain from $C^l, G_\ka^l, G^l$ and $\Lip^l$. Further, if we write, say, $g\in G^3(\R^d)$ without explicitly specifying the codomain of $g$, then it is assumed to be $\R$.

We call a random field 
\[X : \Om \times [0,T]\times \R^d\to \R^d, (\om , t, x) \mapsto X_t(\om)(x)\] 
a solution to a \tSDE{}
\[dX_t = b_t(X_t)\,dt + \si_t(X_t)\,dW_t,\]
without explicit initial value, if $X(x)$ is a the solution to the \tSDE{}
\[dX_t(x) = b_t(X_t(x))\,dt + \si_t(X_t(x))\,dW_t, \quad X_0(x) = x,\]
for all $x\in \R^d$. Similarly, we treat the solution of a recursion
\[\chi_{n+1} = \chi_n + g_n(\chi_n)\]
as a random field $\chi : (\om, n, x) \mapsto \chi_n(\om)(x)$, with $\chi_0(x) = x$.

\section{Expansions in the learning rate}
\subsection{Heuristics}
We heuristically describe how to derive a series expansion of the form \eqref{eq:Xexpans}, as well as \eqref{eq:riskExp}. Details can be found in in the more general setting of Subsection \ref{sec:pertTheory}.
Let $T > 0$ and
\[b^0, b^1, S : [0,T]\times \R \to \R\] 
be measurable functions. We consider the general equation for a second-order diffusion approximation
\begin{equation}
	\label{eq:genericVolSME2}
	dX_t^h = (b^0_t + h b^1_t)(X_t^h)\,dt + \sqrt{h \al_t} S_t(X_t)\,dW_t,
\end{equation}
with $h\in [0,1)$.
We assume that \ref{eq:genericVolSME2} has a unique solution.

Let $g\in C^2(\R)$. We want to, for now heuristically, determine an expression for $\E g(X_t^h)$ using the expansion in \eqref{eq:Xexpans},
\[X^h = X^0 + \sqrt h X^{(1/2)} + h X^{(1)} + h^{3/2} X^{(3/2)} + \cO(h^2).\]
Using a Taylor approximation around the point $X^0$, we get
\begin{align}
	\label{eq:gXexpan}
	g(X^h) = g(X^0) &+ g'(X^0)(\sqrt h X^{(1/2)} + h X^{(1)} + h^{(3/2)} X^{(3/2)}) \nonumber \\
	&+ \frac12 g''(X^0) (\sqrt h X^{(1/2)} + h X^{(1)} + h^{(3/2)} X^{(3/2)})^2 \nonumber\\
	&+ \cO(h^2)\nonumber \\
	= g(X^0) & + \sqrt h g'(X^0) X^{(1/2)}\nonumber \\
	& + h \left(g'(X^0) X^{(1)} + \frac12 g''(X^0)(X^{(1/2)})^2\right)\nonumber \\
	& + h^{3/2}\left(g'(X^0)X^{(3/2)} + g''(X^0) X^{(1/2)}X^{(1)}\right)\nonumber \\
	& + \cO(h^2).
\end{align} 
We can apply the same formula to $b^0, b^1$ and $S$. Plugging the result into \eqref{eq:genericVolSME2}, we get
\begin{align*}
	&d(X^0 + \sqrt h X^{(1/2)} + h X^{(1)} + \cO(h^{(3/2)})) \\
	= & b^0(X^0) + \sqrt h \der b^0(X^0) X^{(1/2)} \\
	& + h \left(b^1{(X^0)} + \frac12 \der^2  b^0(X^0)(X^{(1/2)})^2 + \der b^0(X^0) X^{(1)}\right) + \cO(h^{3/2})\,dt \\
	& +  \sqrt{h \al} S(X^0) + h \sqrt\al \der S(X^0) X^{(1/2)} + \cO(h^{3/2})\,dW,
\end{align*}
where for simplicity we did not consider the $h^{3/2}$-terms.
Thus, by matching powers of $h^{1/2}$ on both sides of the equation, we have
\begin{align}
	\label{eq:XexpanSys}
	dX^0 = & b^0(X^0)\,dt, && X_0^{(0)} = X_0,\nonumber \\
	dX^{(1/2)} = &\der b^0(X^0) X^{(1/2)}\,dt + \sqrt\al S(X^0)\,dW,&& X^{(1/2)}_0 = 0, \nonumber \\
	dX^{(1)} = & b^1(X^0) + \frac12 \der^2 b^0(X^0)(X^{(1/2)})^2 + \der b^0(X^0)X^{(1)}\,dt \nonumber \\
	&+ \sqrt\al \der S(X^0) X^{(1/2)}\,dW, && X_0^{(1)} = 0.
\end{align}
Simplifying further, we heave $\E X^{(1/2)} = 0$ because $\int_0^\blnk \sqrt{\al_t} S(X_T^0)\,dW_t$ is a martingale. In similar fashion one can show that the expectation for the omitted component $X^{(3/2)}$ is zero everywhere.
Further, the quadratic covariation of $X^{(1/2)}$ and $X^{(1)}$ satisfies
\[[X^{(1/2)}, X^{(1)}]_t = \E\int_0^t \al_s S(X_s^{(0)})  \der S(X_s^{(0)}) X^{(1/2)}_s\,ds = 0,\]
and so $\Cov(X^{(1/2)}_t, X^{(1)}_t) = 0$, for all $t\geq 0$.

Moreover, \tIto's formula implies 
\begin{equation}
	\label{eq:X12squared}
	d(X^{(1/2)})^2 = 2\der b^0(X^0) (X^{(1/2)})^2 + \al S(X^0)^2\,dt + 2 X^{(1/2)} \sqrt{\al} S(X^0)\,dW.
\end{equation}
Applying expectation to \eqref{eq:XexpanSys} with the second equation replaced by \eqref{eq:X12squared} yields the system of \tODE s
\begin{align}
\label{eq:perturbSys}
	dX_t^0 = & b^0_t(X_t^0)\,dt, && X_0^{(0)} = X_0, \nonumber\\
	d\Var[X_t^{(1/2)}] = &  2 \der b^0_t(X_t^0)\Var[X_t^{(1/2)}] + \al_t(S(X^0_t))^2 \,dt, && \Var[X_0^{(1/2)}] = 0,\nonumber\\
	d\E[X_t^{(1)}] = & b^1_t(X^0_t) + \frac12 \der^2 b^0_t(X^0_t)\Var[X_t^{(1/2)}]\nonumber \\
	&+ \der b^0_t(X^0_t)\E[X^{(1)}_t]\,dt, && \E[X_0^{(1)}] = 0.
\end{align}
By applying expectation to \eqref{eq:gXexpan} we get
\begin{align}
	\label{eq:EgXh}
	\E[g(X^h)] = g(X^0) &+ h\left(\frac12 g''(X^0)\Var[X^{(1/2)}] +  g'(X^0)\E[X^{(1)}]\right) \nonumber \\ 
	&+ \cO(h^2),
\end{align}
since $\E[X^{(1/2)}] = \E[X^{(3/2)}] = \Cov(X^{(1/2)}, X^{(1)}) = 0$ and $X^0$ is deterministic.
Proposition \eqref{prop:perturbSysSMELike} in Section \ref{sec:pertTheory} shows that our derivation is indeed rigorous under reasonable conditions on the coefficients $b^0, b^1$ and $S$.

\subsection{Perturbation theory for stochastic differential equations}
\label{sec:pertTheory}
We develop a rigorous perturbation theory for \tSDE s depending on a small parameter, to simplify notation in dimension $d = 1$. The results are inspired by \cite{blagoveshchenskii_diffusion_1962}, but geared more towards our desired applications. 

Let $(\Om, \cF_\Om,\P)$ be a complete probability space, $\cF = (\cF_t)_{t\geq 0}$ be a filtration on $(\Om, \cF_\Om,\P)$, satisfying the usual conditions and $W$ be a $\R$-valued $\cF$-\tBm{}. Consider a family of \tSDE s indexed by a small parameter $\ep > 0$,
\begin{equation}
\label{eq:SDEsmallPar}
dY_t^\ep = b^\ep_t(Y_t^\ep)\,dt + \si^\ep_t(Y_t^\ep)\,dW_t,
\end{equation}
driven by $W$.
Our aim is to find random fields $Y^{(0)}, Y^{(1)}, Y^{(2)}, \dots$, such that
\[Y^\ep = Y^{(0)} + \ep Y^{(1)} + \ep^2 Y^{(2)} + \dots\]
Suppose we terminate the series at the level $l$, and we are given random fields
\[Y^{(k)} : \Om \times [0,T]\times \R \to \R, (\om,t,x)\mapsto Y^{(k)}_t(\om)(x),\]
for $k\in \set{0,\dots, l}$. 
We are interested in the remainder term
\[R^\ep := \frac{1}{\ep^{l+1}} \left(Y^\ep - \sum_{k=0}^l Y^{(k)} \ep^k\right).\]
\begin{comment}
\begin{assum}
\label{assum:yTaylor}
We are given (families of) random fields
\begin{align*}
	Y : \Om \times [0,T] \times (0,1)\times \R \to \R, &(\om,t,\ep,x)\mapsto Y^{\ep}_t(\om)(x)\\
	Y^{(k)} : \Om \times [0,T]\times \R \to \R, &(\om,t,x)\mapsto Y^{(k)}_t(\om)(x),
\end{align*}
for $k\in \set{0,\dots, l}$, such that $\nrm{Y^\ep}{*p}\in G(\R)$, uniformly in $\ep \in (0,1)$, for all $p\geq 2$, and $\nrm{Y^{(k)}}{*p} \in G(\R)$, for all $p\geq 2$ and $k\in \set{0,\dots, l}$. Further,
\[\frac{1}{\ep^{l+1}}\nrm{Y^\ep - \sum_{k=0}^l Y^{(k)}\ep^k}{*p} \in G(\R),\]
uniformly in $\ep \in (0,1)$, for all $p\geq 2$.
\end{assum}
\end{comment}
We write $Y^{(\al)} := \prod_{k=1}^l Y^{(\al_k)}$ for every multi-index $\al : \set{0,\dots, l} \to \N_0$.
Note that the multinomial theorem implies
\begin{align*}
\left(\sum_{k=1}^l Y_t^{(k)}\ep^k\right)^n = &\sum_{\stackrel{\al_1,\dots, \al_l}{|\al|=n}} \binom{n}{\al} \prod_{k=1}^l \ep^{k\al_k} (Y_t^{k})^{\al_k} \\
= & \sum_{\stackrel{\al_1,\dots, \al_l}{|\al|=n}} \binom{n}{\al} \ep^{\sum_{k=1}^l k \al_k} Y_t^{(\al)} \\
= & \sum_{k=n}^{nl} Y_t^{(k,n)} \ep^k,
\end{align*}
where
\[Y^{(k,n)} := \sum_{\stackrel{\al_1,\dots, \al_l}{|\al|=n, \sum_j j \al_j = k}} \binom{n}{\al} Y^{(\al)},\]
for $n,k \in \N_0$.

Now, consider a function 
\[b : (0,1)\times [0,T] \times \R \to \R, (\ep, t, y) \mapsto b^\ep_t(y),\]
with $b_t \in C^{l+1}((0,1)\times \R)$ for all $t$.

Write
\[b^{(k)} := \frac{1}{k!} (\der^k_\ep b^\ep)_{|\ep = 0}\]
and
\begin{equation}
\label{eq:bYk}
(b(Y))^{(k)} =  \sum_{m+n\leq k} \frac{1}{n!}\der_y^n b^{(m)}(Y^{(0)}) Y^{(k-m,n)}.
\end{equation}
Note that $(b(Y))^{(k)}$ is the $k$-th coefficient if we expand $b^\ep(Y^\ep)$, or in fact also $b^\ep(Y^{(0)} + \ep Y^{(1)} + \dots \ep^l Y^{(l)})$, into a power series with respect to $\ep$, for any $k\leq l$.
\begin{lem}
\label{lem:bOfSeriesTaylor}
Let $b : (0,1)\times [0,T] \times \R \to \R$ be a function with $b_t \in C^{l+1}((0,1)\times \R)$ for all $t$. Write
\[Z_{n,m}^\ep := \sum_{k=n}^m Y^{(k,n)}\ep^k, \quad n\leq m \in \N\]
Then,
\begin{align*}
b^\ep\left(\sum_{k=0}^l Y^{(k)} \ep^k\right) = &\sum_{k=0}^l (b(Y))^{(k)} \ep^k +\sum_{m+n\leq l} \frac{1}{n!} \der_y^n b^{(m)}(Y^{(0)}) (Z_{n,nl}^\ep - Z_{n,l}^\ep)\ep^m \\
&+\sum_{k=0}^{l+1} \rho_k^\ep(Y^{(0)} + Z_{1,l}^\ep) Z_{(l+1)-k,l(l+1)-lk}^\ep \ep^k,
\end{align*}
where 
\[\rho_k^\ep(y) = \frac{l+1}{k!((l+1)-k)!} \int_0^1 (1-\xi)^l \der^k_\ep \der^{(l+1)-k}_y b^{\xi\ep}((1-\xi)Y^{(0)} + \xi y) \,d\xi.\]
\end{lem}
\begin{proof}
By applying Taylor's theorem to $b$ at the point $(\ep, Y^{(0)})$, we have
\begin{align*}
b^\ep(y) = \sum_{m+n\leq l} \frac{1}{n!} \der_y^n b^{(m)}(Y^{(0)}) (y - Y^{(0)})^n \ep^m + \sum_{k=0}^l \rho_k^\ep(y) (y - Y^{(0)})^{(l+1)-k} \ep^k, \quad y\in \R, \ep \in (0,1).
\end{align*}
Note that
\[(Z_{1,l})^n = Z_{n,nl} = Z_{n,l} + Z_{n,nl} - Z_{n,l},\]
and further
\begin{align*}
\sum_{m+n\leq l} \frac{1}{n!} \der_y^n b^{(m)}(Y^{(0)}) Z_{n,l} \ep^m = & \sum_{m+n\leq l} \sum_{q=n}^l \frac{1}{n!} \der_y^n b^{(m)}(Y^{(0)}) Y^{(q,n)}\ep^{m+q} \\
= & \sum_{k=0}^l \sum_{m+n\leq l} \frac{1}{n!} \der_y^n b^{(m)}(Y^{(0)}) Y^{(k-m,n)}\ep^{k}\\
= & \sum_{k=0}^l (b(Y))^{(k)} \ep^k
\end{align*}
Thus, setting $y := \sum_{k=0}^l Y^{(k)} \ep^k = Y^{(0)} + Z_{1,l}$ shows the result.
\end{proof}

\begin{comment}
and thus
\begin{align*}
b^\ep(Y^\ep) = & \sum_{m,n} \frac{1}{m!n!} \der_y^n b^{(m)}(Y^{(0)}) (Y^\ep - Y^{(0)})^n \ep^m \\
= & \sum_{m,n} \sum_{q\geq n} \frac{1}{m!n!} \der_y^n b^{(m)}(Y^{(0)})  Y^{(q,n)} \ep^{m+q} + \cO(\ep^{l+1})\\
= & \sum_{k\geq 0} \sum_{m+n\leq k}  \frac{1}{m!n!} \der_y^n b^{(m)}(Y^{(0)})  Y^{(k-m,n)} \ep^k + \cO(\ep^{l+1}).
\end{align*}
With this in mind, we set
\begin{align}
\label{eq:bYk}
(b(Y))^{(k)} := & \sum_{m+n\leq k} \frac{1}{m!n!}\der_\ep^m \der_y^n b^\ep(Y^{(0)})|_{\ep = 0} Y^{(k-m,n)}\nonumber\\
=  &\sum_{m+n\leq k}  \frac{1}{m!n!}\sum_{\stackrel{\al_1,\dots, \al_l}{|\al|=n, \sum_j j \al_j = k}} \binom{n}{\al} \der_\ep^m \der_y^n b^\ep(Y^{(0)})|_{\ep = 0} \prod_{k=1}^l Y^{(\al_k)}\\
= & \sum_{m+n\leq k} \frac{1}{m!n!}\sum_{\stackrel{\al_1,\dots, \al_{k-1}}{|\al|=n, \sum_j j \al_j = k}} \binom{n}{\al} \der_\ep^m \der_y^n b^\ep(Y^{(0)})|_{\ep = 0} \prod_{k=1}^l Y^{(\al_k)}\nonumber\\
&+ \sum_{m=0}^k \frac{1}{m!}\der_\ep^m \der_y b^\ep(Y^{(0)})|_{\ep = 0} Y^{(k)}\nonumber
\end{align}
This is simply the $k$-th coefficient if we expanded $b^\ep(Y^\ep)$ into power series with respect to $\ep$.
Note that $b_t\in G((0,1)\times \R)$, uniformly in $t$, if and only, if $b_t^\ep \in G(\R)$ uniformly in $\ep$ and $t$.
\end{comment}
\begin{rem}
\label{rem:computingPerturbSDEs}
Let us compute $(b(Y))^{(k)}$ for $k = 0,1,2,3$.
We have
\[Y^{(k,1)} = Y^{(k)}, Y^{(k,0)} = 0,\quad k \in \N_1.\]
Further,
\[Y^{(2,2)} = (Y^{(1)})^2, Y^{(3,2)} = 2 Y^{(1)} Y^{(2)}, Y^{(3,3)} = (Y^{(1)})^3.\]
Thus, we can write
\begin{align*}
(b(Y))^{(k)} = & \sum_{m+n\leq k} \frac{1}{n!} \der_y^n b^{(m)}(Y^{(0)}) Y^{(k-m,n)} \\
= &b^{(k)}(Y^{(0)}) + \sum_{m=0}^{k-1} \der_y b^{(m)}(Y^{(0)}) Y^{(k-m)} + \frac12 \der_y^2 b^{(k-2)}(Y^{(0)}) (Y^{(1)})^2 1_{[2,\infty)}(k) \\
&+ (\der_y^2 b^{(k-3)}(Y^{(0)}) Y^{(1)} Y^{(2)} +\frac16\der_y^3 b^{(k-3)}(Y^{(0)}) (Y^{(1)})^3) 1_{[3,\infty)}(k)\\
&+ \sum_{\stackrel{m+n\leq k}{m\leq k-4, n \geq 2}} \frac{1}{n!} \der_y^n b^{(m)}(Y^{(0)}) Y^{(k-m,n)} 1_{[4,\infty)}(k)
\end{align*}
In particular,
\begin{align*}
(b(Y))^{(0)} = &b^0(Y^{(0)}), \\
(b(Y))^{(1)} = &b^{(1)}(Y^{(0)}) +  \der_y b^{0}(Y^{(0)}) Y^{(1)},\\
(b(Y))^{(2)} = &b^{(2)}(Y^{(0)}) + \sum_{m=0}^1 \der_y b^{(m)}(Y^{(0)}) Y^{(2-m)} + \frac12 \der_y^2 b^{0}(Y^{(0)}) (Y^{(1)})^2, \\
(b(Y))^{(3)} = &b^{(3)}(Y^{(0)}) + \sum_{m=0}^{2} \der_y b^{(m)}(Y^{(0)}) Y^{(3-m)} + \frac12 \der_y^2 b^{(1)}(Y^{(0)}) (Y^{(1)})^2 \\
&+ \der_y^2 b^{0}(Y^{(0)}) Y^{(1)} Y^{(2)} +\frac16\der_y^3 b^{0}(Y^{(0)}) (Y^{(1)})^3.
\end{align*}
\end{rem}
\begin{prop}
\label{prop:composEpSeries}
Suppose we are given a function $b : (0,1) \times [0,T]\times \R \to \R$, with $b_t \in G^{l+1}((0,1)\times \R)$, uniformly in $t\in [0,T]$, and $b_t^\ep \in \Lip(\R)$, uniformly in $t\in [0,T]$ and $\ep \in (0,1)$.
Then there exist a multivariate polynomial $q\in \R[y_0,\dots, y_{l+1}]$ and a constant $C > 0$, such that
\[\frac{1}{\ep^{l+1}}\left|b^\ep(Y^\ep) - \sum_{k=0}^l (b(Y))^{(k)}\ep^k\right| \leq q(|Y^{(0)}|, \dots, |Y^{(l)}|, |Y^\ep|) + C |R^\ep|.\]
Further, the coefficients of $q$ and the constant $C$ depend only on, and are increasing functions of the $\Lip$- and $G^{l+1}$-norms of $b$.
\end{prop}
In this and similar situations, when we refer to, say, the $\Lip$-norm of $b : I\times \R^d \to \R^d$ with $b_i\in \Lip$, uniformly in $i\in I$, what we really mean is $\sup_{i\in I} \nrm{b_i}{\Lip}$.
\begin{proof}
We write
\begin{align*}
b^\ep(Y^\ep) - \sum_{k=0}^l (b(Y))^{(k)} \ep^k = & b^\ep(Y^\ep) - b^\ep\left(\sum_{k=0}^l Y^{(k)}\ep^k\right) \\
&+ b^\ep\left(\sum_{k=0}^l Y^{(k)}\ep^k\right) - \sum_{k=0}^l (b(Y))^{(k)} \ep^k.
\end{align*}
Then,
\begin{align*}
\left|b^\ep(Y^\ep) - b^\ep\left(\sum_{k=0}^l Y^{(k)}\ep^k\right)\right| \leq & \ep^{l+1}\sup_{\stackrel{\ep\in (0,1)}{t\in [0,T]}} \nrm{b^\ep_t}{\Lip} |R^\ep|.
\end{align*}
On the other hand, recall Lemma \ref{lem:bOfSeriesTaylor}.
\begin{comment}
Set 
\[Z_{n,m}^\ep := \sum_{k=n}^m Y^{(k,n)}\ep^k.\]
Taylor's theorem implies
\begin{align*}
	b^\ep\left(\sum_{k=0}^l Y^{(k)}\ep^k\right) = &\sum_{k=0}^l \sum_{n=0}^{l - k} \der_\ep^k \der_y^n b^\ep(Y^{(0)})|_{\ep = 0}  (Z_{1,l}^\ep)^n \ep^k \\
	&+ \sum_{k=0}^{l+1} \rho_k^\ep(Y^{(0)} + Z_{1,l}^\ep) (Z_{1,l}^\ep)^{(l+1) - k} \ep^k\\
	= & \sum_{k=0}^l (b(Y))^{(k)} \ep^k \\
	&+ \sum_{k=0}^l \sum_{n=0}^{l - k} \der_\ep^k \der_y^n b^\ep(Y^{(0)})|_{\ep = 0}  (Z_{n,nl}^\ep - Z_{n,l}^\ep) \ep^k \\
	&+ \sum_{k=0}^{l+1} \rho_k^\ep(Y^{(0)} + Z_{1,l}^\ep) Z_{(l+1) - k,l(l+1) - lk}^\ep \ep^k,
\end{align*}
where
\[\rho_k^\ep(y) = \frac{l+1}{k!((l+1)-k)!} \int_0^1 (1-\xi)^l \der^k_\ep \der^{(l+1)-k}_y b^{\xi\ep}((1-\xi)Y^{(0)} + \xi y) \,d\xi.\]
\end{comment}
The Taylor remainder satisfies
\[|\rho_k^\ep(y)| \lesssim |\der_\ep^k \der_y^{(l+1)-k} b^{\xi\ep}((1-\xi)Y^{(0)} + \xi y)| \lesssim 1 + |Y^{(0)}|^\ka + |Y^\ep|^\ka ,\]
for some $\ka \in \N$. Thus, $\left|b^\ep\left(\sum_{k=0}^l Y^{(k)}\ep^k\right) - \sum_{k=0}^l (b(Y))^{(k)} \ep^k\right|$ is bounded above by
\begin{align*}
&\ep^{l+1}\sum_{m+n\leq l} \frac{1}{n!} \der_y^n b^{(m)}(Y^{(0)}) \ep^{-(l+1)}|Z_{n,nl}^\ep - Z_{n,l}^\ep| \\
&+ \ep^{l+1}\sum_{k=0}^{l+1} |\rho_k^\ep(Y^{(0)} + Z_{1,l}^\ep)| |Z_{(l+1) - k,l(l+1) - lk}^\ep| \ep^{k-(l+1)},
\end{align*}
where $\ep^{-(l+1)}|Z_{n,nl}^\ep - Z_{n,l}^\ep|$ and $|\rho_k^\ep(Y^{(0)} + Z_{1,l}^\ep)| |Z_{(l+1) - k,l(l+1) - lk}^\ep| \ep^{k-(l+1)}$ are bounded by multivariate polynomials in $|Y^{(0)}|,\dots, |Y^{(k)}|, |Y^\ep|$, not depending on $\ep$ (only on $\nrm{b}{G^l}$).
\end{proof}
Note that if $q \in \R[x_1,\dots, x_l]$ is a multivariate polynomial, i.e.\  we can write
\[q(x) = \sum_{|\al|\leq n} q_\al x^\al,\]
and  $X_1,\dots, X_l$ are stochastic processes, then by Hölder's inequality
\begin{align*}
\nrm{q(X_1,\dots, X_l)}{*p}\leq &\sum_{|\al|\leq n} |q_\al| \nrm{X^\al}{*p} \\
\leq & \sum_{|\al|\leq n} |q_\al| \prod_{k=1}^l \nrm{X_k}{*p \al_k l}^{\al_k}.
\end{align*}

\begin{prop}
\label{prop:perturbSys}
Let $T > 0$ and $l\in \N_0$. Suppose we are given functions
\begin{align*}
b : (0,1)\times[0,T]\times \R \to \R, (\ep, t, x) \mapsto b^{\ep}_t(x), \\
\si : (0,1)\times[0,T]\times \R  \to \R, (\ep, t, x) \mapsto \si^{\ep}_t(x)
\end{align*}
such that $b_t^\ep, \si_t^\ep \in \Lip^{l+1}\cap G_1$, uniformly in $t\in [0,T]$ and $\ep \in (0,1)$. Let $Y$ be a solution of the family of \tSDE s
\begin{equation}
dY_t^\ep = b^\ep_t(Y_t^\ep)\,dt + \si^{\ep}_t(Y_t^{\ep})\,dW_t.
\end{equation}
Then for every $k\leq l$, there exist a unique solution $Y^{(k)}$ of
\begin{equation}
\label{eq:SDEPerturbCoeff}
dY^{(k)}_t = (b(Y))_t^{(k)}\,dt + (\si(Y))_t^{(k)}\,dW_t, \quad Y^{(k)}_0 = \begin{cases}
Y_0, & k = 0,\\
0, & k \in \N, 
\end{cases}
\end{equation}
and the solutions satisfy
\[\nrm{Y^{(0)}}{*p} \in G_1(\R), \quad\nrm{Y^{(k)}}{*p} < \infty, k\in \N,\]
for all $p\geq 2$.
Here, $b(Y)^{(k)}$ and $\si(Y)^{(k)}$ are given by \eqref{eq:bYk}.
Further,
\[\frac{1}{\ep^{l+1}}\nrm{Y^\ep - \sum_{k=0}^l Y^{(k)}\ep^k}{*p} \in G(\R),\]
uniformly in $\ep \in (0,1)$, for all $p\geq 2$. Moreover $\nrm{\nrm{Y^{(0)}}{*p}}{G}$, $\nrm{Y^{(k)}}{*p}$ and $\sup_{\ep\in (0,1)} \nrm{\nrm{Y^\ep - \sum_{k=0}^l Y^{(k)}\ep^k}{*p}}{G}$ depend only on, and are increasing functions of the $\Lip^{l+1}$- and $G_1$-norms of $b$ and $\si$, for all $p\geq 2$.
\begin{comment}
Here, 
\[f^{(k)} = \frac{1}{k!} (\der_\ep^n f^\ep)|_{\ep = 0}, \quad k \leq l,\]
and the derivative $\der^l_\ep$ of $Y$ is taken formally with respect to the formal power series expansion
\[Y_t^\ep = \sum_{k = 0}^\infty Y_t^{(k)} \ep^k.\]
\end{comment}
\end{prop}
Note that even though we initially introduced $Y^{(k)}$ for $k > 0$ as random fields, they in fact do not depend on the initial value assigned to $Y$, in contrast to $Y^{(0)}$.
\begin{proof}
We may write $\nrm{Y^{(k)}}{*p}\in G_1(\R)$ in place of $\nrm{Y^{(k)}}{*p} <\infty$, for $k\in \N$.
Suppose \eqref{eq:SDEPerturbCoeff} has a unique solution for all $k'< k$, such that $\nrm{Y^{(k')}}{*p} \in G_1(\R)$, for all $p\geq 2$. Then we can plug these solutions into \eqref{eq:SDEPerturbCoeff}. The coefficients in \eqref{eq:SDEPerturbCoeff}  are then uniformly linear and Lipschitz in $Y^{(k)}$. Hence, \eqref{eq:SDEPerturbCoeff} has a unique solution, with $\nrm{Y^{(k')}}{*p} \in G_1(\R)$, for all $p\geq 2$.
Similarly, \eqref{eq:SDEsmallPar} has a unique solution $Y^\ep$, with $\nrm{Y^\ep}{*p} \in G_1(\R)$, for all $p\geq 2$.
Now, consider the remainder term
\[R^\ep := \frac{1}{\ep^{l+1}}\left(Y^\ep - \sum_{k=0}^l Y^{(k)}\ep^k\right).\]
Then, by using the \tSDE{} governing $Y^\ep$ and $Y^{(0)},\dots Y^{(l)}$ we have, for all $p\geq 2$ and $t\in [0,T]$,
\begin{align*}
\nrm{R^\ep}{*p,t} \leq & \frac1{\ep^{l+1}} \nrm{\int_0^{\blnk} b_s^\ep(Y_s^\ep) - \sum_{k=0}^l (b(Y))^{(k)}_s\ep^k\,ds}{*p,t} \\
&+ \frac1{\ep^{l+1}} \nrm{\int_0^{\blnk} \si_s^\ep(Y_s^\ep) - \sum_{k=0}^l (\si(Y))^{(k)}_s\ep^k\,dW_s}{*p,t} \\
& \lesssim \frac1{\ep^{l+1}} \int_0^t \nrm{b^\ep(Y^\ep) - \sum_{k=0}^l (b(Y))^{(k)}\ep^k}{*p,s} \,ds \\
&+ \frac1{\ep^{l+1}} \int_0^t \nrm{\si^\ep(Y^\ep) - \sum_{k=0}^l (\si(Y))^{(k)}\ep^k}{*p,s} \,ds\\
&\lesssim \int_0^t (\nrm{q(|Y^{(0)}|, \dots, |Y^{(l)}|, |Y^\ep|)}{*p,s} + C \nrm{R^\ep}{*p,s})\,ds
\end{align*}
for some multivariate polynomial $q$. Then, by Grownall's inequality
\[\nrm{R^\ep}{*p,t} \leq C_1\nrm{q(|Y^{(0)}|, \dots, |Y^{(l)}|, |Y^\ep|)}{*p,t} e^{tC_2},\]
for some constants $C_1, C_2 > 0$, with
\[\nrm{q(|Y^{(0)}|, \dots, |Y^{(l)}|, |Y^\ep|)}{*p,t}\in G(\R).\]
\end{proof}

Let us make a few observations about the series expansion of $Y$ according to \ref{prop:perturbSys} in the special case we encounter for second-order diffusion approximations to stochastic approximations algorithms with a learning rate $h = \ep^2$.
\begin{prop}
\label{prop:perturbSysSMELike}
Suppose we are in the setting of Proposition \ref{prop:perturbSys} with $l = 3$. Further, we assume
\[\si^{(0)} = \si^{(2)} = b^{(1)} = b^{(3)} = 0.\]
%and $b^{(0)}, b^{(2)}, \si^{(1)}, \si^{(3)}$ have linear growth, uniformly in $I\times [0,T]$, i.e.\ there exists a $C > 0$, such that
%\[|\si^{i,(1)}_t(y)|\leq C(1 + |y|),\]
%for all $i\in I, t\in [0,T]$ and $y\in \R$.
Then the following statements hold true.
\begin{enumerate}[(i)]
\item $Y^{(0)}$ is deterministic and $Y^{(1)}$ is Gaussian,
\item $\E[(Y^{(1)})^{2k+1}] = 0$, for all $k\in \N_0$,
\item $\E[Y^{(3)}] = 0$,
\item $\Cov(Y^{(1)}, Y^{(2)}) = 0$.
\end{enumerate}
Further, the following dynamics hold true
\begin{align}
\label{eq:perturbODESys}
dY_t^{(0)} = & b^{(0)}_t(Y_t^{(0)})\,dt, && Y_0^{(0)} = Y_0 \nonumber\\
d\Var[Y_t^{(1)}] = &  2 \der_y b^{(0)}_t(Y_t^{(0)})\Var[Y_t^{(1)}] + \si^{(1)}_t(Y^{(0)}_t)^2  \,dt, && \Var[Y_0^{(1)}] = 0,\nonumber\\
d\E[Y_t^{(2)}] = & b^{(2)}_t(Y^{(0)}_t) + \frac12 \der_y^2 b^{(0)}_t(Y^{(0)}_t)\Var[Y_t^{(1)}]\nonumber \\
	&+ \der_y b^{(0)}_t(Y^{(0)}_t)\E[Y^{(2)}_t]\,dt, && \E[Y_0^{(2)}] = 0.
\end{align}
\begin{comment}
The solutions of the latter two linear equations are given by
\begin{align*}
\Var^{Y_0}[Y_t^{(1)}] = & \int_0^t \ph_{s,t}^2 \si_s^{(1)}(Y_s^{(0)})^2 \,ds \\
\E^{Y_0}[Y_t^{(2)}] = &  \int_0^t \ph_{s,t}\left(b^{(2)}_s(Y^{(0)}_s) + \frac12 \der_y^2 b^{(0)}_s(Y^{(0)}_s)\Var^{Y_0}[Y_s^{(1)}]\right)\,ds,
\end{align*}
where
\[\ph_{s,t} = \exp\left(\int_s^t \der_y b^{(0)}_s(Y_s^{(0)})\,ds\right).\]
\end{comment}
\end{prop}
\begin{proof}
%\todoq{Rewrite carefully for conditional expectation $\E^{Y_0}$}
\textbf{Regarding $Y^{(0)}$:}
Since $\si^{(0)} = 0$, the equation governing $Y^{(0)}$ is the \tODE{}
\[dY^{(0)}_t = b^{(0)}_t(Y_t^{(0)})\,dt, \quad Y^{(0)}_0 = Y_0,\]
by Remark \ref{rem:computingPerturbSDEs}.
In particular, $Y^{(0)}$ is deterministic.\\
\textbf{Regarding $Y^{(1)}$:}
Since $b^{(1)} = 0$ and again by Remark \ref{rem:computingPerturbSDEs}, $Y^{(1)}$ satisfies the linear equation 
\[dY^{(1)}_t = \der_y b^{(0)}_t(Y_t^{(0)}) Y^{(1)}_t\,dt + \si^{(1)}_t(Y^{(0)}_t)\,dW_t,\]
and the diffusion term does not depend on $Y^{(1)}$. Thus, $Y^{(1)}$ is Gaussian. 
Observe that $\left(\int_0^t \si_s^{(1)}(Y_s^{(0)})\,dW_s\right)_{s\in [0,T]}$ is a martingale.
Hence, by the optional stopping theorem
\[d\E[Y_t^{(1)}] = \der_y b^{(0)}(Y_t^{(0)})\E[Y^{(1)}_t]\,dt ,\quad \E[Y_0^{(1)}] = 0.\]
The unique solution to this \tODE s is $\E[Y^{(1)}] = 0$, which proves (ii) for $k = 0$. Assume that (ii) is true for $k-1\geq 0$.
By \tIto's formula, we have
\begin{align*}
d(Y^{(1)}_t)^{k} = & k \der_y b^{(0)}_t(Y_t^{(0)})(Y^{(1)}_t)^{k}\,dt \\
&+ \frac12 k(k-1) \si^{(1)}_t(Y^{(0)}_t)^2(Y^{(1)}_t)^{(k-2)} \,dt\\
&+k \si^{(1)}_t(Y^{(0)}_t)(Y^{(1)}_t)^{(k-1)}\,dW_t.
\end{align*}
Substituting $k$ with $2k+1$ and taking the expectation yields
\begin{align*}
d\E[(Y^{(1)}_t)^{2k+1}] = & k \der_y b^{(0)}_t(Y_t^{(0)})\E[(Y^{(1)}_t)^{2k+1}]\,dt \\
&+ \frac12 k(k-1) (\si^{(1)}_t(Y^{(0)}_t))^2\E[(Y^{(1)}_t)^{2k-1}] \,dt\\
&+k \E[\si^{(1)}_t(Y^{(0)}_t)(Y^{(1)}_t)^{2k}\,dW_t].
\end{align*}
By Hölder's inequality, we have
\begin{align*}
\nrm{\si^{(1)}(Y^{(0)})(Y^{(1)})^{2k}}{2} &\leq  \nrm{|\si^{(1)}(Y^{(0)})||(Y^{(1)})|^{2k}}{2}\\
&\lesssim \nrm{\si^{(1)}(Y^{(0)})}{4} \nrm{(Y^{(1)})^{2k}}{4} \\
&\lesssim (1 + \nrm{Y^{(0)}}{4}) \nrm{Y^{(1)}}{8k}^{2k}\\
&< \infty.
\end{align*}
Thus, 
\[\left(\int_0^t \si^{(1)}(Y^{(0)})(Y^{(1)})^{2k}\,dW\right)_{t\in [0,T]}\]
is a square-integrable martingale, and by optional stopping as well as property (ii) for $k' < k$,
\begin{align*}
d\E[(Y^{(1)}_t)^{2k+1}] = & k \der_y b^{(0)}_t(Y_t^{(0)})\E[(Y^{(1)}_t)^{2k+1}]\,dt, \\ & \E[(Y^{(1)}_0)^{2k+1}] = 0.
\end{align*}
Again, the unique solution to this \tODE{} is $\E[(Y^{(1)})^{2k+1}] = 0$, proving (ii) for general $k$. The equation for $\Var[Y^{(1)}]$ in \eqref{eq:perturbODESys} follows readily.\\
\textbf{Regarding $Y^{(2)}$ and (iv):}
The process $Y^{(2)}$ satisfies the equation
\begin{align*}
dY^{(2)}_t = & b^{(2)}_t(Y^{(0)}_t) +\der_y b^{0}_t(Y^{(0)}_t) Y^{(2)}_t + \frac12 \der_y^2 b^{0}_t(Y^{(0)}_t) (Y^{(1)}_t)^2\,dt\\
&+ \der_y \si_t^{(1)}(Y^{(0)}_t)Y^{(1)}_t\,dW_t.
\end{align*}
Denote by $[X,Y]$ the quadratic covariation of processes $X$ and $Y$. Then
\begin{align*}
\E[[Y^{(1)},Y^{(2)}]_t] = & \int_0^t \E[\si^{(1)}_s(Y^{(0)}_s)\der_y \si_s^{(1)}(Y^{(0)}_s)Y^{(1)}_s]\,ds\\
= & 0,
\end{align*}
by (i) and (ii). Hence, $\Cov(Y^{(1)}, Y^{(2)})$ is $0$ everywhere as well.\\
\textbf{Regarding $Y^{(3)}$:}
The process $Y^{(3)}$ satisfies the equation
\begin{align*}
dY_t^{(3)} = & \der_y b^{0}_t(Y^{(0)}_t) Y^{(3)}_t + \der_y b^{(2)}_t(Y^{(0)}_t) Y^{(1)}_t + \der_y^2 b^{0}_t(Y^{(0)}_t) Y^{(1)}_t Y^{(2)}_t +\frac16\der_y^3 b^{0}_t(Y^{(0)}_t) (Y^{(1)}_t)^3\,dt \\
&+ \si^{(3)}_t(Y^{(0)}_t) + \der_y \si^{(1)}_t(Y^{(0)}_t) Y^{(2)}_t + \frac12 \der_y^2 \si^{(1)}_t(Y^{(0)}_t) (Y^{(1)}_t)^2\,dW_t.
\end{align*}
Because of (ii) and (iv), as well as another optional stopping argument, we have
\begin{align*}
d\E[Y_t^{(3)}] = & \der_y b^{(0)}_t(Y^{(0)}_t) \E[Y^{(3)}_t]\,dt, \E[Y_0^{(3)}] = 0
\end{align*}
with unique solution $\E[Y^{(3)}] = 0$. This proves (iii).
\end{proof}

\begin{prop}
\label{prop:funOfPerturbSeries}
Suppose we are in the setting of Proposition \ref{prop:perturbSysSMELike} and we are given a function $g\in G^4(\R)$. Set
\begin{align*}
Z = &\frac12 \der_y^2 g(Y^{(0)})(Y^{(1)})^2 + \der_y g(Y^{(0)})Y^{(2)}, \\
V^\ep = &\ep \der_y g(Y^{(0)})Y^{(1)} + \ep^3 (g(Y))^{(3)}.
\end{align*}
Then we have $\E[V^\ep] = 0, \ep \in (0,1)$, and
\begin{align*}
r^\ep_{1,p} := \frac{1}{\ep^4}\nrm{g(Y^\ep) - g(Y^{(0)}) - V^\ep - \ep^2 Z}{*p} \in G(\R),
\end{align*}
uniformly in $\ep \in (0,1)$, for all $p\geq 2$.
In particular,
\[r^\ep_2 := \frac{1}{\ep^4}\left|\E g(Y_T^\ep) - \left(g(Y_T^{(0)}) + \ep^2 \left(\frac12 \der_y^2 g(Y^{(0)}_T)\Var[Y^{(1)}_T] + \der_y g(Y^{(0)}_T)\E[Y^{(2)}_T]\right)\right)\right|\]
is in $G(\R)$, uniformly in $\ep \in (0,1)$. Further, $\sup_{\ep \in (0,1)} \nrm{r^\ep_{1,p}}{G}$ and $\sup_{\ep \in (0,1)} \nrm{r^\ep_{2}}{G}$ depend only on, and are increasing functions of the $\Lip$- and $G^{l+1}$-norms of $b$ and $\si$, as well as $\nrm{g}{G^4}$, for all $p\geq 2$.
\end{prop}

\begin{proof}
As a special case of Remark \ref{rem:computingPerturbSDEs} we have
\begin{align*}
(g(Y))^{(0)} = &g(Y^{(0)}), \\
(g(Y))^{(1)} = &\der_y g(Y^{(0)}) Y^{(1)},\\
(g(Y))^{(2)} = &\der_y g(Y^{(0)}) Y^{(2)} + \frac12 \der_y^2 g(Y^{(0)}) (Y^{(1)})^2, \\
(g(Y))^{(3)} = &\der_y g(Y^{(0)}) Y^{(3)} + \der_y^2 g(Y^{(0)}) Y^{(1)} Y^{(2)} +\frac16\der_y^3 g(Y^{(0)}) (Y^{(1)})^3.
\end{align*}
Thus,
\[\sum_{k=0}^3 (g(Y))^{(k)}\ep^k = g(Y^{(0)}) + V^\ep + \ep^2 Z.\]
From Proposition \ref{prop:perturbSysSMELike} we know that $\E[V^\ep] = 0$.
Propositions \ref{prop:composEpSeries} and \ref{prop:perturbSys} imply $r_{1,p}^\ep \in G(\R),$
uniformly in $\ep \in (0,1)$, for all $p\geq 2$.
Then, it follows readily that $r_2^\ep\in G(\R)$, uniformly in $\ep \in (0,1)$.
\end{proof}
\subsection{Perturbation theory for optimal control of stochastic differential equations}
\label{sec:pertTheoryOptCont}
Proposition \ref{prop:funOfPerturbSeries} ends with a statement on how the polynomial growth constant of a remainder term $r_2^\ep$
depends on various norms, each depending on $b,\si$ and $g$. Similar statements can be found throughout the section. The purpose of these statements is the ability to extend the approximation result to discuss optimal control problems, in which the coefficients of  \eqref{eq:SDEsmallPar} depend on the choice of control. From \ref{prop:funOfPerturbSeries} we can immediately deduce the following.
\begin{cor}
\label{cor:funOfPerturbSeriesIdx}
Let $I$ be a set and $T > 0$. Suppose we are given functions
\begin{align*}
	b : I \times (0,1) \times [0,T]\times \R\times  \to \R, (i, \ep, t, x) \mapsto b^{i,\ep}_t(x), \\
	\si : I \times (0,1) \times [0,T]\times \R \to \R, (i, \ep, t, x) \mapsto \si^{i,\ep}_t(x)
\end{align*}
such that $b_t^{i,\ep}, \si_t^{i,\ep} \in \Lip^4\cap G_1$, uniformly in $i\in I$, $t\in [0,T]$ and $\ep \in (0,1)$, and 
\[\si^{(0)} = \si^{(2)} = b^{(1)} = b^{(3)} = 0.\]
Let $Y$ be the unique solution of the family of \tSDE s (omitting $i$)
\begin{equation}
	dY_t^{\ep} = b^{\ep}_t(Y_t^{\ep})\,dt + \si^{\ep}_t(Y_t^{\ep})\,dW_t, \quad Y^{(0)}\in \R,
\end{equation}
and $(Y^{(0)}, \Var[Y^{(1)}], \E[Y^{(2)}])$ be the unique solution of the family of systems of \tODE s
\begin{align}
dY_t^{(0)} = & b^{(0)}_t(Y_t^{(0)})\,dt, && Y_0^{(0)} = Y_0 \nonumber\\
d\Var[Y_t^{(1)}] = &  2 \der_y b^{(0)}_t(Y_t^{(0)})\Var[Y_t^{(1)}] + \si^{(1)}_t(Y^{(0)}_t)^2  \,dt, && \Var[Y_0^{(1)}] = 0,\nonumber\\
d\E[Y_t^{(2)}] = & b^{(2)}_t(Y^{(0)}_t) + \frac12 \der_y^2 b^{(0)}_t(Y^{(0)}_t)\Var[Y_t^{(1)}]\nonumber \\
	&+ \der_y b^{(0)}_t(Y^{(0)}_t)\E[Y^{(2)}_t]\,dt, && \E[Y_0^{(2)}] = 0.
\end{align}
Then for  every $g\in G^4(\R)$, there exists a $C\in G(\R)$, with
\[\sup_{i\in I}\left|\E g(Y_T^{i,\ep}) - \left(g(Y_T^{i,(0)}) + \ep^2 \frac12 \der_y^2 g(Y^{i,(0)}_T)\Var[Y^{i,(1)}_T] + \der_y g(Y^{i,(0)}_T)\E[Y^{i,(2)}_T]\right)\right| \leq C \ep^4\]
for all $\ep \in (0,1)$.
\end{cor}
As a consequence of Corollary \ref{cor:funOfPerturbSeriesIdx} we may transfer deterministic control problems between $Y$ and $(Y^{(0)}, \Var[Y^{(1)}], \E[Y^{(2)}])$.
\begin{cor}
\label{cor:funOfPerturbSeriesOptim}
In the setting of Corollary \ref{cor:funOfPerturbSeriesIdx} the following holds true. For every $g\in G^4(\R)$, which is bounded from below, there exists a $C\in G(\R)$ with
\[\left|\inf_{i\in I}\E g(Y_T^{i,\ep}) - \inf_{i\in I}\left(g(Y_T^{i,(0)}) + \ep^2 \left(\frac12 \der_y^2 g(Y^{i,(0)}_T)\Var[Y^{i,(1)}_T] + \der_y g(Y^{i,(0)}_T)\E[Y^{i,(2)}_T]\right)\right)\right| \leq C \ep^4,\]
for all $\ep \in (0,1)$.
\end{cor}
\begin{proof}	Note that for functions $f,g : I \to \R$, bounded from below, we have
	\[|\inf f - \inf g| \leq \sup|f-g|.\]
	Hence, the result follows from Corollary \ref{cor:funOfPerturbSeriesIdx}.
\end{proof}

\section{Second-order diffusion approximations for SGD}
\label{sec:proofSME2}
In this section we prove a general second-order approximation result for stochastic gradient descent and similar algorithms in higher dimensions. Our approximating equations extends the second-order stochastic modified equation in \cite{li_stochastic_2019} by allowing for time-dependent drift and diffusion coefficients, e.g.\ learning rate or batch size schedules. Moreover, we formulate all our results in such a way that we can apply the diffusion approximation to study optimal control problems (e.g.\ see the last sentence in Theorem \ref{thm:2ndOrderSME}).

\subsection{Main result}
Let $(\Om, \cF_\Om,\P)$ be a complete probability space. Consider a random function
\[f : \Om \times [0,1] \times [0,T]\times \R^d\to \R^d, (\om, h, t, x) \mapsto f_t^{h}(\om)(x),\]
such that $(f_t)_{t\in [0,T]}$ is an independent family.
Let $\cF = (\cF_t)_{t\geq 0}$ be a filtration on $(\Om, \cF_\Om,\P)$ independent of $f$, satisfying the usual conditions and $W$ be an $\R^d$-valued $\cF$-\tBm{}. We consider a parameter $h\in (0,1)$, which acts as discretization parameter or maximal learning rate and is essential in describing the diffusion approximation.

Given an initial value $x\in \R^d$ define the stochastic \emph{one-step method} with \emph{increment function} $f$ by
\begin{equation}
	\label{eq:GSGD}
	\chi_{n+1}^{h} = \chi_n^h + h f_{nh}^{h}(\chi_n^{h}), \quad \chi_0 = x.
\end{equation}

\begin{assum}
	\label{assum:H}
	There exists a random variable $Z$ with with finite moments, such that
	\[|f_t^{h}(x)| \leq Z(1 + |x|), a.s.,\]
	for all $h\in [0,1], t\in [0,T]$ and $x\in \R^d$.
\end{assum}

Further, define
\[\bar f : [0,1]\times [0,T]\times \R^d \to \R^d, (h,t,x) \mapsto \E f_t^{h}(x).\] 
and 
\[V : [0,1]\times[0,T]\times \R^d\to \R^{d\times d}, (h,t,x) \mapsto \E[(f_t^{h}(x) - \bar f_t^{h}(x))^{\otimes 2}].\] 
Here $\label{eq:outSq}z^{\otimes 2} = zz^\dagger$ for any $z\in \R^d$. Since $V$ is positive semi-definite and symmetric, a unique matrix square root $\sqrt V$ exists everywhere. By Assumption \assref{assum:H} we have $\bar f^{h}, \sqrt V^{h} \in G_1([0,T]\times\R^d)$, uniformly in $h$.
\begin{assum}
\label{assum:barHandSi}
We have $\bar f^{h}_t \in \Lip^4$ and $\sqrt V^{h}_t \in \Lip^3$, uniformly in $h$ and $t$, with $\bar f^{h} \in C^{1,4}([0,T]\times \R^d)$ and $\sqrt V^{h}\in C^{0,3}([0,T]\times \R^d)$  for all $h$.
Further, $\der_t f^{h}_t \in G_1 \cap \Lip^3$, uniformly in $h$ and $t$, and $\nrm{g^{h}}{\Lip^\tme}\in G(\R^d)$, uniformly in $h$ , for all $g\in \set{\bar f, \nabla \bar f, \der_t \bar f, \sqrt V}$.
\end{assum}
The conditions on $\bar f$ ensure that the drift coefficient in Equation \ref{eq:2ndOrderSME} below satisfies
\[\bar f_t^{h} - \frac 1 2 h (\nabla \bar f_t^{h} \bar f_t^{h} + \der_t \bar f_t^{h})\in G_1 \cap \Lip^3,\]
uniformly in $h$ and $t$.

The relevance of not assuming that  $\sqrt V$ is differentiable in time is that for volatility control problems it allows optimal controls which are not differentiable, which frequently occur by imposing bounds on the controls.

For all $h \in (0,1)$ we consider the family of \tSDE s
\begin{equation}
	\label{eq:2ndOrderSME}
	dX_t^{h} = \left(\bar f_t^{h}(X_t^{h}) - \frac 1 2 h (\nabla \bar f_t^{h} \bar f_t^{h} + \der_t \bar f_t^{h})(X_t^{h})\right)\,dt +\sqrt{h V_t^{h}}(X_t^{h})\,dW_t,
\end{equation}
where $\nabla g : [0,T]\times \R^d \to \R^{d\times d}$ denotes the Jacobian of a function $g : [0,T]\times \R^d \to \R^d$ in the space variable, i.e.\ $(\nabla g)_{i,j} = \der_{x_j} f_i$ for all $i,j\in \set{1,\dots, d}$. Crucially, observe the occurrence of the $\der_t \bar f$ term in \eqref{eq:2ndOrderSME}. It vanishes if $\bar f$ is constant in $t$. Therefore, this term was not present in previous works such as \cite{li_stochastic_2019}. To exhibit this term we use an \tIto{}-Taylor approximation for a time-inhomogeneous SDEs (cf.\ Proposition \ref{prop:itoTaylor2} and Remark \ref{rem:2ndOrderSMEvsSGD}).
\begin{comment}
In this section the solution to an SDE $X$ depends on both the time $t\in [0,T]$ as well as the initial value $x\in \R^d$. For example given functions $f_1, f_2 : \R^d \to \R$ and $t\in [0,T]$ the equation
\[f_1 = f_2(X_t)\]
means
\[f_1(x) = f_2(X_t(x)),\]
for all $x\in \R^d$, and similarly for the SGD iterations $\GSGD^h$.
\end{comment}
\begin{satz}
\label{thm:2ndOrderSME}
Assume \assref{assum:H} and \assref{assum:barHandSi}. For all $h \in (0,1)$ let $X^h$ be the solution of \eqref{eq:2ndOrderSME}. Then for all $g\in G^3(\R^d)$ and $T > 0$, there exists a $C\in G(\R^d)$, such that 
\[\max_{n\in \set{0,\dots, \floor{T/h}}} |\E g(\GSGD_n^{h}) - \E g(X_{nh}^{h})|\leq Ch^2,\]
for all $h\in (0,1)$. Further, $\nrm{C}{G}$ depends only, and is an increasing function of $\nrm{g}{G}$, $\nrm{Z}{\ka}$ for some large $\ka \in \N$, and
\begin{itemize}
\item $\sup_{\stackrel{t\in [0,T]}{h\in (0,1)}} \nrm{\bar f^h_t}{\Lip^4}, \sup_{\stackrel{t\in [0,T]}{h\in (0,1)}} \nrm{\sqrt V^h}{\Lip^3}$, $\sup_{\stackrel{t\in [0,T]}{h\in (0,1)}} (\nrm{\der_t \bar f^h_t}{\Lip^3} + \nrm{\der_t \bar f^h_t}{G_1})$,\\
\item $\sup_{h\in (0,1)} \nrm{\nrm{\tilde g^{h}}{\Lip^\tme}}{G}$, for all $\tilde g\in \set{\bar f, \nabla \bar f, \der_t \bar f, \sqrt V}$.
\end{itemize}
\end{satz}
The proof of Theorem \ref{thm:2ndOrderSME} is postponed to Subsection \ref{sec:proofSME2proof}.

\subsection{Diffusion approximations for optimal control}
Similar to Subsection \ref{sec:pertTheoryOptCont},
Theorem \ref{thm:2ndOrderSME} ends with a statement on how the polynomial growth constant of
\[h^{-2}\max_{n\in \set{0,\dots, \floor{T/h}}} |\E g(\GSGD_n^{h}) - \E g(X_{nh}^{h})|\]
depends on various norms, each depending on $f$ and $g$. 

Consider now an index set $I$ and an $I$-indexed family of random functions 
\[f : \Om \times I \times [0,1] \times [0,T]\times \R^d\to \R^d, (\om, h, t, x) \mapsto f_t^{i,h}(\om)(x).\]
Suppose every statement in \assref{assum:H} and \assref{assum:barHandSi} holds, \emph{uniformly} in $i\in I$. Then we can directly deduce the following.
\begin{cor}
\label{cor:2ndOrderSMEIdx}
For all $h \in (0,1)$ and $i\in I$ let $X^{i,h}$ be the solution of the \tSDE{}
\begin{equation}
\label{eq:2ndOrderSMEIdx}
dX_t^{i,h} = \left(\bar f_t^{i,h}(X_t^{i,h}) - \frac 1 2 h (\nabla \bar f_t^{i,h} \bar f_t^{i,h} + \der_t \bar f_t^{i,h})(X_t^{i,h})\right)\,dt +\sqrt{h V_t^{i,h}}(X_t^{i,h})\,dW_t.
\end{equation} Then for all $g\in G^3(\R^d)$ and $T > 0$, there exists a $C\in G(\R^d)$, such that 
\[\sup_{i\in I} \max_{n\in \set{0,\dots, \floor{T/h}}} |\E g(\GSGD_n^{i,h}) - \E g(X_{nh}^{i,h})|\leq Ch^2,\]
for all $h\in (0,1)$.
\end{cor}
As a consequence of \ref{cor:2ndOrderSMEIdx} we may transfer deterministic control problems between the one-step method $\chi$ and its diffusion approximation.
\begin{cor}
\label{cor:2ndOrderSMEOptim}
For all $h \in (0,1)$ let $X$ be the solution of \eqref{eq:2ndOrderSMEIdx}. Then for all $g\in G^3(\R^d)$, which are bounded from below, and $T > 0$, there exists a $C\in G(\R^d)$, such that 
\[\max_{n\in \set{0,\dots, \floor{T/h}}} |\inf_{i\in I}\E g(\GSGD_n^{i,h}) - \inf_{i\in I} \E g(X_{nh}^{i,h})|\leq Ch^2,\]
for all $h\in (0,1)$.
\end{cor}
In the following remark we give simple conditions for SGD, featuring a learning rate- and a (continuous) batch size schedule, to satisfy \assref{assum:barHandSi} \emph{uniformly} in the choice of schedules.
\begin{rem}
	\label{rem:simplAssum}
	Let $L > 0$ and consider the following index set of pairs consisting of a learning rate control and a volatility control
	\begin{align*}
		I = &\set{u : [0,T] \to [0,1] : u\in C^1, \nrm{u}{\Lip}, \nrm{\der_t u}{\Lip} \leq L} \\
		&\times \set{\al : [0,T]\to [0,1] : \nrm{\sqrt \al}{\Lip} \leq L}.
	\end{align*}
	Suppose there exist functions $H : \R^d \to \R^d$ and $S : \R^d\to \R^{d\times d}$, such that
	\[\bar f_t^u(x) = u_t H(x), \quad V_t^{\al}(x) = \al_t S(x),\]
	satisfying
	\[H \in G_1\cap \Lip^4, \sqrt S \in G_1 \cap \Lip^3.\]
	Then,
	\begin{align*}
		|\der^\al \bar f_t(x) - \der^\al \bar f_t(y)| \leq & \nrm{H}{\Lip^4} |x-y|, \quad |\al|\leq 4,\\
		|\bar f_t(x) - \bar f_s(x)| \leq & L \nrm{H}{G_1} |t-s| (1 + |x|),\\
		|\nabla \bar f_t(x) - \nabla \bar f_s(x)| \leq & L \nrm{\nabla H}{\infty} |t-s| \\
		= & L \nrm{H}{\Lip} |t-s|, \\
		|\der_t \bar f_t(x) - \der_t \bar f_s(x)| \leq & L\nrm{H}{G_1} |t-s|( 1 + |x|), \quad |\al|\leq 3,\\
		|\der^\al\der_t \bar f_t(x) - \der^\al \der_t \bar f_t(y)| \leq & L \nrm{H}{\Lip^3}|x-y|, \quad |\al|\leq 3,\\
		|\der^\al \sqrt{V_t(x)} - \der^\al \sqrt{V_t(y)}| \leq & \nrm{\sqrt S}{\Lip^3}|x-y|, \quad |\al|\leq 3,\\
		|\sqrt{V_t(x)} - \sqrt{V_s(x)}| \leq & L \nrm{\sqrt{S}}{G_1} |t-s| (1 + |x|),
	\end{align*}
for all $x,y\in \R^d$ and $s,t\in [0,T]$.
	Hence, $\bar f$ and $\sqrt V$ satisfy Assumption \assref{assum:barHandSi}, uniformly in $(u,\al) \in I$.
\end{rem}

\subsection{Results from stochastic analysis}
Here we collect minor extensions to well known results from stochastic analysis to make the proofs of our main results self-contained. We consider \tSDE s with coefficients
\[b : [0,T] \times \R^d \to \R^d, S : [0,T]\times \R^d\to \R^{d\times d}.\]
\begin{satz}
\label{thm:SDESol}
Suppose $b_t,S_t \in G_1 \cap \Lip$, uniformly in $t$. Then, for every $p\geq 2, T > 0$ and random field $\ph : \Om \times [0,T]\times \R^d \to \R^d$ with $\nrmps{\ph}{p} < \infty$, the \tSDE{}
\[dX_t = b_t(X_t)\,dt + S_t(X_t)\,dW_t, \quad X_0 = \ph\]
admits a unique\footnote{Of course, we mean unique up to indistinguishability.} solution $X$ on $[0,T]$, such that the family of solutions $X = (X_t)_{t\geq 0}$ satisfies
\[\nrmps{X}{p} \lesssim 1 + \nrmps{\ph}{p}.\]
The constant factor on the RHS depends only on, and is an increasing function of the $G_1$- and $\Lip$- norms of $b$ and $S$.
\end{satz}

\begin{proof}
This essentially a standard result, cf.\ \cite{Kunita2004} Theorem 3.1 and 3.2 for example. The extension to from an initial value $x\in \R^d$ to a process $\ph$ is discussed in \cite{li_stochastic_2019} Theorem 18 and 19.
\end{proof}

\begin{satz}
\label{thm:SDEDer}
Let $l\in \N, p\geq 1$ and suppose $b_t,S_t \in G_1\cap \Lip^l$, uniformly in $t$. Let $x\in \R^d, s\in [0,T]$ and $X$ be the unique solution to the family of \tSDE s
\[dX_t = b_t(X_t)\,dt + S_t(X_t)\,dW_t.\]
Then $X$ is $l$-times continuously differentiable w.r.t. to the initial condition $x$ at any $(t,x)\in [s,T]\times \R^d$, a.s. and for every multi-index $\al$ with $0 < |\al| \leq l$, $\der^\al X$ satisfies the \tSDE{}
\[\der^\al X_t = \psi_\al + \int_s^t \nabla b_u(X_u)\der^\al X_u\,du + \int_s^t \nabla S_u(X_u)\der^\al X_u\,dW_u,\]
where $\nrmps{\psi_\al}{p}\in G(\R^d)$ for all $p\geq 2$. Moreover,
\[\E(\der^\al X_t) = \der^\al \E(X_t),\]
for all $t\geq 0$. Further, $\nrm{\nrmps{\psi_\al}{p}}{G}$ depends only on, and is an increasing function of the $G_1$- and $\Lip^l$-norms of $b$ and $S$.
\end{satz}
\begin{proof}
For the proof cf.\ \cite{Kunita2004} Theorem 3.4. More specifically, for every $l\in \N$, assuming the result holds for all $l' < l$ define
\begin{align*}
	Y := (X, \der_1 X, \dots \der_d X, \der_{1,1} X, \dots, \der_{1,d} X, \der_{2,1} X, \dots, \der_{d,\dots, d} X)^\dagger,
\end{align*}
where the last partial derivative is of the order $l-1$. Then $Y$ satisfies the \tSDE{}
\begin{align*}
	Y = &\mat{x\\e_1\\\vdots \\0} + \mat{0\\\psi_1\\\vdots \\\psi_{d,\dots, d}} + \int_s^t\mat{b_u(X_u)\\ \nabla b_u(X_u)\der_1 X_u \\ \vdots \\ \nabla^{l-1} b_u(X_u)\der_{d,\dots, d} X_u}\,du \\
	&+ \int_s^t\mat{S_u(X_u)\\ \nabla S_u(X_u)\der_1 X_u \\ \vdots \\ \nabla^{l-1} S_u(X_u)\der_{d,\dots, d} X_u}\,dW_u,
\end{align*}
where the processes $\psi_1,\dots, \psi_{d,\dots, d}$ consists of additional integrals $\int_s^t\,du$ and $\int_s^t\,dW_u$ of the remaining terms induced by repeated application of the chain rule. The terms within $\int_s^t \,du$ and $\int_s^t \,dW_u$ respectively are seen to be functions of $u$ and the state $Y$, satisfying the conditions of \cite{Kunita2004} Theorem 3.4. By applying it again to the SDE governing $Y$ the result follows via induction on $l$.
\end{proof}

\begin{prop}
	\label{prop:EgXPolyGrowth}
	Let $l\in \N, p\geq 1$ and $b_t,S_t \in G_1\cap \Lip^l$, uniformly in $t$. Let $X$ be the unique solution to the family of \tSDE s 
	\[dX_t^s(x) = b_t(X_t^s(x))\,dt + S_t(X_t^s(x))\,dW_t, \quad X_s^s(x) = x.\]
	and $g : \R^d \to \R\in G^l(\R^d)$. Define
	\[v^s_t(x) := \E g(X^s_t(x)), \quad x\in \R^d.\]
	Then $v^s_t\in G^l(\R^d)$, uniformly in $s$ and $t$. Further, $\sup_{s\leq t}\nrm{v_t^s}{G^l}$  depends only on, and is an increasing function of the $G_1$- and $\Lip^l$-norms of $b$ and $S$, as well as the $G^l$-norm of $g$.
\end{prop}
\begin{proof}
	Let $\al$ be a multi-index with $|\al|\leq l$. By induction one can show $\E \der^\al g(X) = \der^\al \E g(X)$ using Theorem \ref{thm:SDEDer}. By the higher chain rule,
	\begin{align*}
		|\der^\al v^s_t| = &\E|\der^\al g(X_t^s)| \leq \sum_{j=1}^{|\al|} \nrmps{\nabla^j g(X)}{2} \sum_{\cB\in \cS_j^\al} N(\al, \cB)\prod_{\be \in \cB} \nrmps{\der^\be X}{2\#\cB}.
	\end{align*}
Here,
\[\nrmps{\nabla^j g(X)}{2} = \nrmps{\sqrt{\sum_{|\be|\leq j} |\der^\be g(X)|^2}}{2}.\]
Further, $\cS_j^\al$ is the set of all partitions of $\al$ into $j$ multi-set multi-indices (each partition being a multi-set as well), $N(\al, \cB)\in \N$, $\#\cB$ is the size of the partition and the product $\prod_{\be \in \cB}$ respects the multiplicities of $\be \in \cB$.
	From $g\in G^l(\R^d)$ and Theorem \ref{thm:SDEDer} we conclude $\der^\al v\in G(\R^d)$.
\end{proof}

\subsection{Moment estimates and growth conditions}
We collect various moment estimates for SGD-like algorithms and their approximating SDEs in this section.

\subsubsection{Stochastic Gradient Descent}
Recall the definition of $\GSGD$ in \eqref{eq:GSGD}, as well as Assumption \assref{assum:H}. Denote the stochastic one-step methods iterations starting at time $n$ with initial value $x\in \R^d$ and parameter $h\in (0,1)$ by $\GSGD_n^{h,n}(x)$. 
Given a discrete process $Y$, e.g.\ $Y = \GSGD^{h,k}(x)$, we write

\newcommand{\diff}{\Delt}
\begin{equation}
\label{eq:diff}
\Delt Y_n := Y_{n+1} - Y_n.
\end{equation}
We let $\diff Y_n^h := \diff Y_n^{h,0}$. Observe that $\diff Y_n^{h,n}(x) = Y_{n+1}^{h,n}(x) - x$.
\begin{lem}
\label{lem:momentsSGD}
We have
\begin{align*}
\E \diff \GSGD_n^{h,n} = & h \bar f_{nh}, \\
\E (\diff \GSGD_n^{h,n})^{\otimes 2} = & h^2 (V_{nh} + \bar f_{nh}^{\otimes 2}).
\end{align*}
\end{lem}

\begin{proof}
Straightforward. 
%We have
%\begin{align*}
%\E \diff \GSGD_n^{h,n} = &\E(\eta_n^h H_{\ga_n}) = \eta_n^h \mH,\\
%\E (\diff \GSGD_n^{h,n})^{\otimes 2} = &(\eta_n^h)^2(H_{\ga(n)} - \mH + \mH)^{\otimes 2})\\
% = &(\eta_n^h)^2 (\vH + \mH^{\otimes 2}).
%\end{align*}
\end{proof}

\begin{lem}
\label{lem:linGrowthEstSGD}
Let $p\geq 1$. The following estimates hold true: 
\begin{enumerate}[(i)]
\item For every $T > 0$  there exists a constant $C > 0$, such that
\[\sup_{h\in(0,1)}\nrmps{\GSGD^h(x)}{p,\floor{\frac T h}} \leq C(1 + |x|),\]
for $x\in \R^d$, and $C$ depends only on, and is an increasing function of $\nrm{Z}{p}$.
\item We have \[\nrm{\diff \GSGD_n^{h,i,n}(x)}{p} \leq h \nrm{Z}{p}(1 + |x|),\]
for all $h\in (0,1), i\in I, n\in \N$ and $x\in \R^d$. 
\end{enumerate}
\end{lem}
\begin{proof}
\begin{enumerate}[(i)]
\item  Let $p\in \N$.
For every $h \in (0,1)$ and $n\in \set{0,\dots, \floor{T/h}}$,
\[\nrmps{(\GSGD^h)}{p,n} =  \sup_{i\in I} \left(\E \max_{n'\in \set{-1,\dots,n-1}} |\GSGD_{n'+1}^{h,i}|^p\right)^{1/p}.\]
We have
\begin{align*}
	|\GSGD^h_{n+1}|^p \leq & |\GSGD_n^h + h f_{ nh}^h(\GSGD_n^h)|^p\\
	\leq &|\GSGD^h_n|^p + \sum_{k=1}^p \binom{p}{k} |\GSGD^h_n|^{p-k} h^k |f_{nh}^h(\GSGD_n^h)|^k,
\end{align*}
for all $n\in \set{0,\dots, \floor{T/h}}$.
Now, for $k\in \set{1,\dots p}$, $h\in (0,1)$ and $n\in \set{0,\dots, \floor{T/h}}$,
\begin{align*}
	\nrmps{(|\GSGD^h|^{p-k}|f_{\blnk h}^h(\GSGD^h)|^k)}{1,n} &\leq \nrmps{(|\GSGD^h|^{p-k}Z^k(1 + |\GSGD^h|)^k)}{1,n}\\
	&\leq \E[Z^k] \nrmps{(|\GSGD^h|^{p-k} + |\GSGD^h|^{k + p - k})}{1,n} \\
	&\leq 2\E[Z^k](1 + \nrmpsp{(\GSGD^h)}{p,n}{p})\\
\end{align*}
using the inequalities $y^p + y^q \leq 2(1 + y^q)$ for $0<p\leq q$ and $y \geq 0$, as well as Assumption \assref{assum:H}.
Therefore, if we let $\chi_{-1} = 0$,
\begin{align*}
	\nrmpsp{(\GSGD^h)}{p,n+1}{p}  \leq &\E\max_{n'\in \set{-1,\dots, n}} |\GSGD_{n'}^{h}|^p \\
	&+ \E\max_{n'\in \set{-1,\dots, n}}\sum_{k=1}^p\binom{p}{k} h^k|\GSGD_{n'}^{h}|^{p-k}|f_{n'h}^{h}(\GSGD_{n'}^{h,i})|^k\\
	\leq&\nrmpsp{(\GSGD^h)}{p,n}{p} + \sum_{k=1}^p\binom{p}{k} h^k\nrmps{|\GSGD^h|^{p-k}|f_{\blnk h}^{h}(\GSGD^h)|^k)}{1,n}\\
	\leq&\nrmpsp{(\GSGD^h)}{p,n}{p}+ Ch(1  + \nrmpsp{(\GSGD^h)}{p,n}{p})\\
	=&(1 + Ch)\nrmpsp{(\GSGD^h)}{p,n}{p} + Ch,
\end{align*}
where $C := \sum_{k=1}^p \binom{p}{k}\E[|Z|^k]$.
By induction over $n$,
\[\nrmpsp{(\GSGD^h)}{p,n}{p} \leq (1 + Ch)^n\nrmpsp{(\GSGD^h)}{p,0}{p}  + Ch\left(\sum_{k=0}^{n-1}(1 + Ch)^k\right),\]
for all $h\in (0,1)$ and $n\in \set{0,\dots, \floor{T/h}}$. Consequently,
\begin{align*}
	\nrmpsp{\GSGD^h(x)}{p,\floor{\frac T h}}{p}&\leq (1 + Ch)^{\floor{\frac T h}} |x|^p + Ch \sum_{k=0}^{\floor{\frac T h}} (1 + Ch)^k\\
	&\leq (1 + Ch)^{\frac T h} |x|^p + Ch\frac T h(1 + Ch)^{\frac T h}\\
	&=(CT + |x|^p)e^{\log(1+Ch)\frac T h} \\
	&\leq (CT + |x|^p)e^{CT},
\end{align*}
for all $h\in (0,T)$ and $x\in \R^d$, since $\log(1 + y) \leq y$ for all $y > -1$. Now, the inclusion follows for $p\in \N$.
For arbitrary $p\geq 1$ we have $\nrmps{Y}{p} \leq \nrmps{Y}{\ceil{p}}$ and thus the result is proven.
\item We have
\[\nrmp{\diff \GSGD_n^{h,n}(x)}{p} = \nrmp{h f_{nh}^{h}(x)}{p} \leq h \nrm{Z}{p} (1 + |x|),\]
for all $h\in (0,1), i\in I$ and $x\in \R^d$.
\end{enumerate}
\end{proof}
\subsubsection{Diffusion Approximations}
We shall now consider moments and growth conditions for solutions of (families of) \tSDE s that will act as approximations to SGD. 

\newcommand{\Img}{\operatorname{Im}}
Given the family of solutions $X$ to a \tSDE{}, we define the family of discrete processes
\newcommand{\discX}{\tilde X}
\begin{equation}
\label{eq:discX}
\tilde X_n^{h}(x) := X_{nh}^{h}(x),
\end{equation}
with $h\in (0,1), x\in \R^d$ and $n\in \set{0,\dots, \floor{T/h}}$. Then,
\[\diff \discX_n^{h,n}(x) = X_{nh}^h(x) - x.\]
\begin{lem}
\label{lem:linGrowthEstSDE}
Let
\[b : (0,1)\times [0,T] \times \R^d \to \R^d, S : [0,T] \times \R^d \to \R^{d\times d} \in G_1(\R^d) \cap \Lip,\]
uniformly in $t$ and $h$, and $X$ be the unique solution to the family of \tSDE s
\[dX_t^h= b_t^h(X_t^h)\,dt + \sqrt{h}S_t(X_t^h)\,dW_t.\]
Then for all $p\geq 2$ there exists a $C\in G(\R^d)$, such that
\[\nrm{\diff \discX_n^{h,n}}{p} \leq h C,\]
for all $h\in (0,1)$ and $n\in \set{0,\dots, \floor{T/h}}$. Further, $\nrm{C}{G}$  depends only, and is an increasing function of the $G_1$- and $\Lip$-norms of $b$ and $S$. 
\end{lem}
\begin{proof}
We have
\[\nrm{\diff \discX_n^{h,n}}{p} \leq \nrm{\int_{nh}^{(n+1)h} b^h_s(X_s^h)ds}{p} + \sqrt h \nrm{\int_{nh}^{(n+1)h} S_s(X_s^h)\,dW_s}{p}.\]
On the one hand
\begin{align*}
\nrm{\int_{nh}^{(n+1)h} b_t^h(X_t^h)dt}{p}\leq & h^{1-\frac 1 p} \left(\int_{nh}^{(n+1)h} \E|b^h_t(X_t^h)|^p\,dt\right)^{1/p} \\
\leq & h\left(\E\sup_{t\in [0,T]} |b_t^h(X_t^h)|^p\right)^{1/p}\\
\leq & h \nrmps{b^h(X^h)}{p}.
\end{align*}
By Theorem \ref{thm:SDESol}, and in particular by the last sentence, we have
\[x\mapsto \nrmps{b^h(X^h(x))}{p}\in G(\R^d),\]
uniformly in $h$. An analogous statement is true for $S$.
On the other hand,
\begin{align*}
\sqrt h\nrm{\int_{nh}^{(n+1)h} S_t(X_t^h)\,dW_t}{p} \leq & \sqrt{\frac{p(p-1)}{2}} h^{1 - \frac 1 p} \nrm{S(X^h)}{p}\\
\leq & c_1  h \nrmps{S(X^h)}{p},
\end{align*}
for some $c_1 > 0$, where we have used \tIto's isometry and Jensen's inequality.
\end{proof}

\begin{prop}
\label{prop:weakOneStepDiff}
Let $l\in \N$, $k\in \set{0,\dots, \floor{T/h}}$,
\[b : (0,1) \times [0,T] \times \R^d \to \R^d, S :[0,T] \times \R^d \to \R^{d\times d} \in G_1(\R^d) \cap \Lip^{l+1},\]
uniformly in $h,t$, and let $X$ be the unique solution to the family of \tSDE s
\[dX_t^h = b_t^h(X_t^h)\,dt + \sqrt{h}S_t(X_t^h)\,dW_t.\]
Suppose further we are given $\ka \in \N$,
\[g : (0,1) \times \N \times \R^d\to \R, (h,k, x) \mapsto g^h_k(x) \in G^{l+1}_\ka (\R^d),\]
uniformly in $k$ and $h$,
and assume there exists a function $C \in G(\R^d)$ such that
\begin{align*}
|\E (\diff \GSGD_k^{h,k})^\al - \E (\diff \discX_k^{h,k})^\al| \leq & h^{l+1} C, |\al| \leq l\\
\nrm{\diff \GSGD_k^{h,k}}{(2l+2)\vee \ka}^{l+1}, \nrm{\diff \discX_k^{h,k}}{(2l+2)\vee \ka}^{l+1} \leq &h^{l+1} C,
\end{align*}
for all $h\in (0,1)$ and $k\in \set{0,\dots, \floor{T/h}}$.
Then there exists a function $C'\in G(\R^d)$, such that
\[|\E g_k^h(\GSGD_{k+1}^{h,k}) - \E g_k^h(\discX_{k+1}^{h,k})| \leq h^{l+1} C',\]
for all $h\in (0,1)$ and $k\in \set{0,\dots, \floor{T/h}}$. Further, $\nrm{C'}{G}$ depends only on, and is an increasing function of $\nrm{C}{G}$ and $\nrm{g}{G^{l+1}}$.
\end{prop}
%In the applications of Proposition \ref{prop:weakOneStepDiff}, the families $b,S$ and $g$ may have different index sets $I_1$ and $I_2$ respectively. This is not a problem, since we can consider the common index set $I := I_1\times I_2$.
\begin{proof}
By Taylor's theorem there exist $\thet_{\diff \GSGD_k^{h,k}},\thet_{\diff \discX_k^{h,k}}\in (0,1)$ for every $h\in (0,1)$ and $k$, such that
\begin{align*}
g_k(\GSGD_{k+1}^{h,k}) - g_k(\discX_{k+1}^{h,k}) = & g_k(\GSGD_{k+1}^{h,k}) - g_k - (g_k(\discX_{k+1}^{h,k}) - g_k)\\
=&\sum_{0<|\al|\leq l} \frac 1 {\al!}\der^\al g_k\cdot ((\diff \GSGD_k^{h,k})^\al - (\diff \discX_k^{h,k})^\al) \\
&+\sum_{|\be| = l+1} \sum_{D \in \diff \GSGD_k^{h,k}, \diff \discX_k^{h,k}} \frac 1 {\be !} \der^\be g_k(\blnk + \thet_D D) D^\be
\end{align*}
Since $g_k^h\in G^{l+1}(\R^d)$, uniformly in $k$ and $h$, there exists a $C\in G(\R^d)$, such that
\begin{align*}
|\E[\der^\be g(x + \thet_{D^h} D^h(x))D^h(x)^\be]| \leq &\sup_{\stackrel{h\in (0,1)}{t\in [0,T]}}\nrm{\der^\be g_t^h}{G_\ka}(1 + 2^{\ka-1} |x|^\ka + 2^{\ka-1} \nrm{D^h(x)}{2\ka}^\ka)\\
&\cdot \nrm{D^h(x)}{2l+2}^{l+1} \\
\lesssim & (1 + |x|^\ka + C(x))h^{l+1} C(x),
\end{align*}
for $|\be| = l + 1$ and $D\in \diff \GSGD, \diff \discX$. Therefore,
\begin{align*}
|\E g^h_k(\GSGD_{k+1}^{h,k}(x)) - \E g^{h,i}_k(\discX_{k+1}^{h,k}(x))| \lesssim& \sum_{0< |\al|\leq l} \sup_{\stackrel{h\in (0,1)}{t\in [0,T]}}\nrm{\der^\al g_t^h}{G_\ka} (1+|x|^\ka)h^{l+1}C(x)\\
& + \sum_{|\be| = l+1}\sup_{\stackrel{h\in (0,1)}{t\in [0,T]}}\nrm{\der^\be g_t^h}{G_\ka}(1 + |x|^\ka +  C(x))h^{l+1}C(x).
\end{align*}
\end{proof}

\begin{prop}
\label{prop:oneStep2ManyStep}
Let $l\in \N$ and fix a function $g : \R^d \to \R \in G^{l+1}(\R^d)$. Suppose $X$ is given as in Proposition \ref{prop:weakOneStepDiff}. %and for every family \[g : (0,1)\times I \times\N\times \R^d\to \R, (h,i, k, x) \mapsto g^{h,i}_k(x) \in G^{l+1}(\R^d),\] there exists a function $C\in G(\R^d)$, such that
%\[|\E g^{h,i}_k(\GSGD^{h,k}_{k+1}) - \E g^{h,i}_k(\discX^{h,k}_{k+1})|\leq h^{l+1}C,\]
%for all $i\in I$, $h\in (0,1)$ and $k\in \set{0,\dots, \floor{T/h}}$.
Further, let 
\[g.P_{k,n}^{h}(x) := \int_{\R^d} g(y)\,P_{k,n}^{h}(x, dy) = \E g(\discX_n^{h,k}(x)),\]
where $P^h$ is the transition kernel of $(n,\discX_n^{h})_n$.
Suppose there exists a function $C\in G(\R^d)$, such that
\begin{equation}
\label{eq:gPDiffEst}
|\E g.P_{k,n}^{h}(\GSGD_{k+1}^{h,k}) - \E g.P_{k,n}^{h}(\discX_{k+1}^{h,k})| \leq h^{l+1}C,
\end{equation}
for all $h\in (0,1)$ and $k \in \set{0,\dots, \floor{T/h}}$. Then there exists a function $C'\in G(\R^d)$, such that
\[\max_{n\in \set{0,\dots, \floor{T/h}}} |\E g(\GSGD_n^{h}) - \E g(\discX_n^{h})| \leq h^l C'\]
on $\R^d$. Further, $\nrm{C'}{G}$ depends only on, and is an increasing function of the $G_1$- and $\Lip^l$-norms of $b$ and $S$, the $G^l$-norm of $g$, the $G_\ka$-norm of $C$, if finite, and $\nrm{Z}{\ka}$.
\end{prop}
\begin{proof}
By Proposition \ref{prop:EgXPolyGrowth}, and in particular the last sentence, we have
\[g.P : (k,n,h,x) \mapsto g.P_{k,n}^{h}(x) \in G^{l+1}(\R^d),\]
uniformly in $k,n$ and $h$. Given $n\in \set{0,\dots, \floor{T/h}}$, $\E g(\discX_n) - \E g(\GSGD_n)$ equals
\begin{align*}
&\sum_{k=1}^{n-1} (\E g(\discX_n^{k-1}\GSGD_{k-1}) - \E g(\discX_n^k \GSGD_k)) + \E g(\discX_n^{n-1}\GSGD_{n-1}) - \E g(\GSGD_n) \\
=&\sum_{k=1}^{n-1} \E\E(g(\discX_n^k\discX_k^{k-1}\GSGD_{k-1})|\discX_k^{k-1}\GSGD_{k-1}) - \E \E (g(\discX_n^k \GSGD_k)|\GSGD_k) \\
&+ \E g.P_{n,n}(\discX_n^{n-1}\GSGD_{n-1}) - \E g.P_{n,n}(\GSGD_n)\\
=&\sum_{k=1}^{n} (\E g.P_{k,n}(\discX_k^{k-1}\GSGD_{k-1}) - \E g.P_{k,n}(\GSGD_k)),
\end{align*}
Hence, \eqref{eq:gPDiffEst} and Lemma \ref{lem:linGrowthEstSGD} imply
\begin{align*}
|\E g(\discX_n^h) - \E g(\GSGD_n^h)| \leq & \sum_{k=1}^{\floor{\frac T h}} h^{l+1}\E C(\GSGD_{k-1}^h) \leq h^l T C',
\end{align*}
for some $C'\in G(\R^d)$, all $h\in (0,1)$ and $n\in \set{0,\dots, \floor{T/h}}$, since
\begin{align*}
\E C(\GSGD_{k-1}^h) \leq & \nrm{C}{G_\ka}(1 + \E|\GSGD_{k-1}^h|^\ka) \leq \nrm{C}{G_\ka}\left(1 + \sup_{h\in(0,1)}\nrmpsp{\GSGD}{\ka, \floor{T/h}}{\ka}\right) \\
\lesssim & 1 + |\GSGD_0|^\ka,
\end{align*}
for some $\ka \in \N$, all $h\in (0,1)$ and $k\in \set{0,\dots, \floor{T/h}}$.
\end{proof}

\subsection{Proof of the second-order diffusion approximation}
\label{sec:proofSME2proof}

The next lemma gives a Lipschitz-in-time-like condition for a family of processes $(f_t(X_t(x)))_{t\in [0,T], x\in \R^d}$, where $X$ is the solution of an SDE with Lipschitz coefficients of, say, linear growth. 
\begin{lem}
\label{lem:LipFunAppliedToLipField}
Let $p\geq 2$ and $X : \Om\times [0,T]\times \R^d\to \R^d$ be a random field with $\nrm{X}{\Lip_p^\tme}\in G(\R^d)$ and $\nrm{X_t}{p} \in G(\R^d)$, uniformly in $t$.
Further, let $f : [0,T]\times \R^d\to \R^d$ be a function, with $\nrm{f}{\Lip^\tme} \in G(\R^d)$ and $f_t\in \Lip(\R^d)$, uniformly in $t$.
Then $\nrm{f(X)}{\Lip_p^\tme}\in G(\R^d)$.
\end{lem}
\begin{proof}
Let $C := \nrm{f}{\Lip^\tme}$. We have
\begin{align*}
\nrm{f_t(X_t) - f_s(X_s)}{p} \leq & \nrm{f_t(X_t) - f_s(X_t)}{p} + \nrm{f_s(X_t) - f_s(X_s)}{p} \\
							\leq &\nrm{C(X_t)}{p}(t-s) + \nrm{f_s}{\Lip} \nrm{X_t - X_s}{p} \\
							\lesssim & (t-s)(1+ |x|^\ka), \quad 0 \leq s\leq t\leq T,
\end{align*}
for some $\ka > 0$.
\end{proof}
Given $u,v\in \R^d$ and $A,B\in \R^{d\times d}$ we write
\[\innp u v := \sum_{j=1}^d u_j v_j, \quad \innp A B := \sum_{i,j=1}^d A_{i,j} B_{i,j}\]
in the following.
\begin{prop}
\label{prop:itoTaylor2}
Let\[b^0, b^1 : [0,T]\times \R^d \to \R^d, \si : [0,T]\times \R^d \to \R^{d\times d}\]
be in $\Lip(\R^d)$ and $G_1(\R^d)$, uniformly in time. 
Further, assume $b^0 \in G^{1,2}([0,T]\times \R^d) $ and $b^1, \si \in G^{0,1}([0,T]\times \R^d)$, such that $\nrm{\der_t b^0}{\Lip^\tme}, \nrm{b^1}{\Lip^\tme}, \nrm{\si}{\Lip^\tme}\in G(\R^d)$.
Let $n\in \set{0,\dots, \floor{T/h}-1}$ and $X = (X_t(x))_{t\in [nh,(n+1)h], x\in \R^d}$ be the solution to the family of \tSDE s
\begin{equation}
	\label{eq:genericSME2}
	dX_t(x) = b_t^0(X_t(x)) + h b_t^1(X_t(x))\,dt + \sqrt h \si_t(X_t(x))\,dW_t, \quad X_{nh}(x) = x,
\end{equation}
with $t\in [nh, (n+1)h]$, and $g \in G^3(\R^d)$. Then,
\begin{align*}
	\E g(X_{(n+1)h}) = g + & h \innp{\nabla g}{b^0_{nh}} + \frac{h^2}{2} (\innp{\nabla g}{\nabla b^0_{nh} b^0_{nh} + 2 b^1_{nh} + \der_t b^0_{nh}}) \\
	+ &\frac{h^2}{2} \innp{\nabla^2 g}{\si_{nh}^\dagger \si_{nh} + (b^0_{nh})^{\otimes 2}} + h^3 C
\end{align*} 
for all $h\in (0,1)$, for some $C\in G(\R^d)$. The function $C$ only depends on, and is an increasing function of \begin{itemize}
\item $\sup_{t\in [0,T]} \nrm{b^0_t}{\Lip}, \sup_{t\in [0,T]} \nrm{b^1_t}{\Lip}, \sup_{t\in [0,T]} \nrm{\si_t}{\Lip}$,
\item $\nrm{\der_t b^0}{\Lip^\tme},  \nrm{b^1}{\Lip^\tme}, \nrm{\si}{\Lip^\tme}$,
\item $\nrm{\der_t^k \der^\al b^0}{G}, k=0,1,|\al|\leq 2$; $\nrm{\der^\al b^1}{G}, \nrm{\der^\al \si}{G}, |\al|\leq 1$,
\end{itemize}
and $\nrm{g}{G^3}$.
\end{prop}
\begin{proof}
\tIto's formula implies
\begin{align*}
	g(X_{(n+1)h}) = g(X_{nh}) & + \int_{nh}^{(n+1)h} \innp{\nabla g(X_u) }{b^0_u(X_u)} + h \innp{\nabla g(X_u)}{b_u^1(X_u)}\,du \\
	& + \frac h 2  \int_{nh}^{(n+1)h}  \innp{\nabla^2  g(X_u)}{(\si_u^\dagger \si_u)(X_u)}\,du + R_1,
\end{align*}
where 
\[R_1 := \int_{nh}^{(n+1)h} \innp{\nabla g(X_u)}{\si_u(X_u)}\,dW_u.\]
Note that $\E[R_1] = 0$, by Hölder's inequality, polynomial growth and optional stopping.
Using Einstein's summation convention,
a further application of \tIto's formula yields that 
\[\int_{nh}^{(n+1)h} \innp{\nabla g(X_u)}{b^0_u(X_u)}\,du = \int_{nh}^{(n+1)h} \der_i g(X_u)b^0_u(X_u)^i\,du\]
equals
\begin{align*}
	& \int_{nh}^{(n+1)h} \innp{\nabla g(X_{nh})}{b^0_{nh}(X_{nh})}\,du\\ 
	& + \int_{nh}^{(n+1)h} \int_{nh}^u \innp{\nabla g(X_v)}{\der_t b^0_v(X_v)}\,dv\,du \\
	& + \int_{nh}^{(n+1)h} \int_{nh}^u (\der_{ij} g (X_v) b^0_v(X_v)^i + \der_i g(X_u) \der_j b^0_v(X_v)^i)(b^0_v(X_v))^j\,dv\,du \\
	& + h\int_{nh}^{(n+1)h} \int_{nh}^u (\der_{ij} g (X_v) b^0_v(X_v)^i + \der_i g(X_u) \der_j b^0_v(X_v)^i)(b^1_v(X_v))^j\,dv\,du \\
	& +  \frac h 2 \int_{nh}^{(n+1)h} \int_{nh}^u  \der_{jk} (\der_i g(X_u)b^0_u(X_u)^i) (\si_u^\dagger \si_u)(X_v)^{jk}\,dv\,du\\
	& + \int_{nh}^{(n+1)h} \int_{nh}^u (\der_{ij} g (X_v) b^0_v(X_v)^i + \der_i g(X_u) \der_j b^0_v(X_v)^i)\si_v(X_v)_k^j\,dW_v^k\,du.
\end{align*}
Note that
\[(\der_{ij} g b^0_v(X_v)^i + \der_i g(X_u) \der_j b^0_v(X_v)^i)(b^0_v(X_v))^j = \innp{\nabla^2 g}{b^0_v(X_v)^{\otimes 2}} + \innp{\nabla g}{(\nabla b^0_vb^0_v)(X_v)}.\]
By Lemma \ref{lem:LipFunAppliedToLipField}, we have
\begin{align*}
\nrm{\innp{\nabla g(X)}{(\nabla b^0 b^0)(X) + \der_t b^0(X)} + \innp{\nabla^2 g}{b^0(X)^{\otimes 2}}}{\Lip^\tme_p} \in G(\R^d).
\end{align*}
Further, setting
\begin{align*}
	Z := & h\int_{nh}^{(n+1)h} \int_{nh}^u (\der_{ij} g (X_v) b^0_v(X_v)^i + \der_i g(X_u) \der_j b^0_v(X_v)^i)(b^1_v(X_v))^j\,dv\,du \\
	& +  \frac h 2 \int_{nh}^{(n+1)h} \int_{nh}^u  \der_{jk} (\der_i g(X_u)b^0_u(X_u)^i) (\si_u^\dagger \si_u)(X_v)^{jk}\,dv\,du\\
	& + \int_{nh}^{(n+1)h} \int_{nh}^u (\der_{ij} g (X_v) b^0_v(X_v)^i + \der_i g(X_u) \der_j b^0_v(X_v)^i)\si_v(X_v)_k^j\,dW_v^k\,du,
\end{align*}
we have 
\[\nrm{Z(x)}{p} \leq h^3 C'(x)\]
for some $C'\in G(\R^d)$.
To summarize,
\begin{align*}
	\E \int_{nh}^{(n+1)h} \innp{\nabla g(X_u)}{b^0_u(X_u)}\,du = & h \innp{\nabla g(X_{nh})}{b^0_{nh}} \\
	& + \frac{h^2}{2} (\innp{\nabla g}{\nabla b^0_{nh} b^0_{nh} + \der_t b^0_{nh}} + \innp{\nabla^2 g}{(b^0_{nh})^{\otimes 2}}) \\
	& + h^3 C,
\end{align*}
for some $C\in G(\R^d)$ and all $h\in (0,1)$.
Similarly,
\begin{align*}
	h \E\int_{nh}^{(n+1)h} \innp{\nabla g(X_u)}{b_u^1(X_u)} = & h^2  \innp{\nabla g}{b_{nh}^1} + h^3 C',\\
	\frac h 2  \E \int_{nh}^{(n+1)h}  \innp{\nabla^2  g(X_u)}{(\si_u^\dagger \si_u)(X_u)}\,du = & \frac{h^2}{2} \innp{\nabla^2  g}{\si_{nh}^\dagger \si_{nh}} + h^3 C'
\end{align*}
for some $C'\in G$. In total, we get
\begin{align*}
	\E g(X_{(n+1)h}) = g + & h \innp{\nabla g}{b^0_{nh}} + \frac{h^2}{2} (\innp{\nabla g}{\nabla b^0_{nh} b^0_{nh} + 2 b^1_{nh} + \der_t b^0_{nh}}) \\
	+ &\frac{h^2}{2} \innp{\nabla^2 g}{\si_{nh}^\dagger \si_{nh} + (b^0_{nh})^{\otimes 2}} + h^3 C
\end{align*} 
for all $h\in (0,1)$, for some $C\in G(\R^d)$.
\end{proof}

\begin{rem}
	\label{rem:2ndOrderSMEvsSGD}
	Consider the setting of Proposition \ref{prop:itoTaylor2}. First, set
	\[g(z) := (z-x)_l, l \in \set{1,\dots, d}.\]
	Then $g(x) = 0, \nabla g(x)_j = \delt_{j,l}, \nabla^2 g(x) = 0$ and for any $v\in \R^d$,
	\[\innp{\delt_{\blnk,l}}{v} = v_l.\]
	Recall, $\Delt \tilde X_n^{h,n}(x) = X_{(n+1)h}^{nh}(x) - x$.
	By applying Proposition \ref{prop:itoTaylor2} for all $l\in \set{1,\dots, d}$, we get
	\[\E[\Delt \tilde X_n^{h,n}] = h b^0_{nh} + \frac{h^2}{2} (\nabla b^0_{n h} b^0_{nh} + 2 b^1_{nh} + \der_t b^0_{nh}) + h^3 C,\]
	for all $h\in (0,1)$ and some $C\in G$.
	Similarly, consider now
	\[g(z) := (z-x)_k (z-x)_l, \quad k,l \in \set{1,\dots, d}.\]
	Then
	\[g(x) = 0, \nabla g(x) = 0, \nabla^2 g(x)_{i,j} = \delt_{i,k} \delt_{j,l} + \delt_{i,l}\delt_{j,k},\]
	and for any $A\in \R^{d\times d}$,
	\[\innp{\nabla^2 g(x)}{A} = A_{k,l} + A_{l,k}.\]
	Thus,
	\[\E[(\Delt \tilde X_n^{h,n})^{\otimes 2}] = h^2 (\si^\dagger_{nh} \si_{nh} + (b_{nh}^0)^{\otimes 2}) + h^3 C,\]
	for all $h\in (0,1)$ and some $C\in G$.
	
	Recalling Lemma \ref{lem:momentsSGD}, we have
	\begin{align*}
		\E \diff \GSGD_k^h - \E \diff \discX_n^{h,n} =& h(\bar f_{nh} - b^0_{nh}) + \frac 1 2 h^2 (2b^1_{nh} + (\nabla b^0 b^0)_{nh} + \der_t b^0_{nh}) + h^3 C,\\
		\E (\diff \GSGD_k^h)^{\otimes 2} - \E (\diff \discX_n^{h,n})^{\otimes 2} = &h^2(V - \si^\dagger\si + \bar f^{\otimes 2}- (b^0)^{\otimes 2})_{nh} + h^3 C.
	\end{align*}
	This tell us how to choose the coefficients $b^0,b^1$ and $\si$, such that all terms, except $h^3 C$, vanish. We set
	\[b^0 := \bar f, \quad b^1 := -\frac 1 2 \left(\nabla \bar f \bar f + \der_t \bar f \right), \quad \si := \sqrt V.\]
	
	Note that assumptions \assref{assum:H} and \assref{assum:barHandSi}
	are enough to satisfy the assumptions of Proposition \ref{prop:itoTaylor2} for all $h\in (0,1)$ and $n\in \set{0,\dots, \floor{T/h}}$.
\end{rem}

\begin{comment}

\begin{rem}
\label{rem:2ndOrderSMEvsSGDOld}
With Lemma \ref{prop:itoTaylor2} and \ref{lem:momentsSGD} we may compare SGD with the solution of the family of SDE's
\[dX_t^h = (b^0_t +  h b^1_t)(X_t^h)\,dt + \sqrt{h}S_t(X_t^h)\,dW_t, \quad X_{nh}^h = x.\]
We have
\begin{align*}
\E \diff \GSGD_k^h - \E \diff \discX_n^{h,n} =& h(\bar f_{nh} - b^0_{nh}) + \frac 1 2 h^2 (2b^1_{nh} + (\nabla b^0 b^0)_{nh} + \der_t b^0_{nh}) + h^3 C,\\
\E (\diff \GSGD_k^h)^{\otimes 2} - \E (\diff \discX_n^{h,n})^{\otimes 2} = &h^2(\si^2 - S^\daggerS + \bar f^{\otimes 2}- (b^0)^{\otimes 2})_{nh} + h^3 C.
\end{align*}
This tell us how to choose the coefficients $b^0,b^1$ and $S$, such that all terms but $h^3 C$ vanish. We set
\[b^0 := \bar f, \quad b^1 := -\frac 1 2 \left(\nabla \bar f \bar f + \der_t \bar f \right), \quad S := \si.\]

Note that assumptions \assref{assum:H} and \assref{assum:barHandSi}
are enough to satisfy the assumptions of Lemma \ref{prop:itoTaylor2} for all $h\in (0,1), i\in I$ and $n\in \set{0,\dots, \floor{T/h}}$.
\end{rem}
\end{comment}
We are finally ready to prove Theorem \ref{thm:2ndOrderSME}.
\begin{proof}[Proof of Theorem \ref{thm:2ndOrderSME}]
	By Remark \ref{rem:2ndOrderSMEvsSGD}
	\[|\E(\diff \GSGD_n^{h,n})^\al - \E(\diff \discX_n^{h,n})^\al| \leq h^3 C,\]
	for $|\al|\leq 2$, and by Lemma \ref{lem:linGrowthEstSGD} and \ref{lem:linGrowthEstSDE}
	\[\nrm{\diff \GSGD_n^{h,n}}{p}^3 \vee  \nrm{\diff X_n^{h,n}}{p}^3 \leq h^3 C\]
	for all $n\in \set{0,\dots, \floor{T/h}}, h\in (0,1), p \geq 2$ and some $C\in G(\R^d)$. Denote by $P^{h}$. the transition kernel of $(n, X^{h}_{nh})_{n\in \set{0,\dots, \floor{T/h}}}$ Given any $g\in G^3(\R^d)$, by applying Proposition \ref{prop:weakOneStepDiff} to $\tilde g_n^h := g.P_{k,n}^h$, we have
	\[\left|\E g.P_{k,n}^h(\GSGD_{k+1}^{h,k}) - \E g.P_{k,n}^h(X_{(k+1)h}^{h,kh})\right|\leq h^3C\]
	for some $C\in G(\R^d)$, for all $k\leq n$. Since $\nrm{C}{G}$ is an increasing function of the norms of the coefficients of $X$, as well as $\nrm{Z}{\ka}$, for some large $\ka$, we can choose $C$ independent of $k$.
	Then, by Proposition \ref{prop:oneStep2ManyStep} together with Lemma \ref{lem:linGrowthEstSGD} and Proposition \ref{prop:EgXPolyGrowth},
	\begin{align*}
		 \max_{n\in \set{0,\dots, \floor{T/h}}}|\E g(X_{nh}^{h}) - \E g(\GSGD_n^{h})|\leq & h^2 C
	\end{align*}
	for some $C\in G(\R^d)$ and all $h\in (0,1)$.
\end{proof}

\section{Optimal volatility control}
\label{sec:optCont}
In this section we derive and optimal volatility control for generic equations of the form \eqref{eq:genericVolSME2}. We make use of the Pontryagin maximum principle to solve the optimal batch size control problem (cf.\ \cite{pham2009continuous} Chapter 6.4 for more details).

Recall again equation \eqref{eq:genericVolSME2}
\begin{equation}
	dX_t^h = (b^0_t + h b^1_t)(X_t^h)\,dt + \sqrt{h \al_t} S_t(X_t^h)\,dW_t. \nonumber
\end{equation} 
We make the following assumption on the coefficients of \ref{eq:genericVolSME2}.
\begin{assum}
\label{assum:riskAndGlobExGen}
We have $b^0_t, b^1_t, S_t \in G_1\cap \Lip^4 \text{uniformly in } t$, $S(x)\in C^1([0,T])$ for all $x\in \R$, and $S > 0$ everywhere. Further, the volatility control $\al$ is Lipschitz continuous.
\end{assum}
Assumption \assref{assum:riskAndGlobExGen} ensures that Equation \eqref{eq:genericVolSME2} has a unique solution $X^h$ for all $h\in (0,1)$. Consider an objective function $g : \R \to (0,\infty)$. 
\begin{assum}
	\label{assum:gNice}
	We have $g\in C^2$ with $g''(X_T^0) > 0$ and
	\[g(x) \lesssim 1 + |x|^2, x\in \R.\]
\end{assum}
Note again that the gradient flow $X^0$ does not depend on the batch size.
Thus, based on\eqref{eq:perturbSys} and \eqref{eq:EgXh}, we consider the objective
\begin{equation}
	\label{eq:diffControProbO2}
	\amin_{\al\in A(L)}  \frac12 g''(X_T^0) \Var[X^{(1/2), \al}_T] + g'(X_T^0)\E[X_T^{(1), \al}] + \la \int_0^T \frac{1}{\al_t}\,dt,
\end{equation}
where 
\begin{align}
	d\Var[X^{(1/2),\al}_t] = & 2 B_t^1 \Var[X^{(1/2),\al}_t] + \al_t \si_t^2 \,dt, \label{eq:varx12}\\
	d\E[X_t^{(1),\al}] = & \frac12 B_t^2 \Var[X^{(1/2),\al}_t] + B_t^1 \E[X_t^{(1),\al}] + b^1_t(X_t^0) \,dt \label{eq:ex1},
\end{align}
with $\si_t := S_t(X_t^0)$ and $B^k_t = \der_x^k b^{0}(X^0_t)$. Equivalently, setting
\begin{align*}
	\mu^\al = &\mat{\Var[X^{(1/2),\al}]\\ \E[X^{(1),\al}]}, A = \mat{2B^1 &  0\\ \frac12 B^2 & B^1}, \beta(a) = \mat{a \si^2 \\ b^1(X^0)},
\end{align*}
we have 
\[d\mu_t^\al = A_t \mu_t^\al + \beta_t(\al_t)\,dt,\]
and then the cost at the terminal time $T$ is $\mu\mapsto G^\dagger \mu$, where 
\[G = \mat{\frac12g''(X_T^0)\\g'(X_T^0)}.\]
The Hamiltonian for the control problem is given by
\[\cH(t,m,y,a) = m^\dagger A^\dagger_t y + \beta_t(a)^\dagger y + \frac{\la}{a}.\]
We have
\[0 = \der_a \cH(t,m,y,a) = \si_t^2 y_1 - \la\frac{1}{a^2}.\]
if and only, if
\[a = \sqrt{\frac{\la}{y_1 \si_t^2}},\]
assuming $y_1 > 0$.
Hence,
\begin{equation}
	\label{eq:hamiltMin}
	\amin_{a\in [0,1]} \cH(t,\mu,y,a) =  \sqrt{\frac{\la}{y_1 \si_t^2}} \wedge 1.
\end{equation}
Further,
\[\nabla_m \cH(t,m, y,a) = A_t^\dagger y\]
and so the backward equation is given (in forward form) by
\begin{equation}
	\label{eq:BSDE}
	dY_t = A_t^\dagger Y_t \,dt, \quad Y_T = G
\end{equation}
Hence, its solution is
\[Y_t = \exp\left(-\int_t^T A_s^\dagger\,ds\right) G.\]
Note, that the matrix exponential of any upper triangular $2\times 2$-matrix satisfies
\[\exp\mat{a & b \\ 0 & d} =
\mat{e^a & b\eta \\ 0 & e^d},\]
with
\[\eta = \begin{cases}
	\frac{e^a - e^d}{a-d}, & a\neq d,\\
	e^a, & a = d
\end{cases}.\]
Therefore,
\begin{align}
	\label{eq:BSDEsol}
	Y_t  =  & \mat{e^{-2\be_{t,T}^1} & -\frac12 \be_{t,T}^2\eta_{t,T} \\
		0 & e^{-\be_{t,T}^1}} G = \mat{\frac12 e^{-2\be_{t,T}^1} g''(X_T^0) - \frac12 \be_{t,T}^2 \eta_{t,T} g'(X_T^0)\\ e^{-\be_{t,T}^1} g'(X_T^0)},
\end{align}
where
\[\be^k_{t,T} = \int_t^T B_s^k \,ds,\]
and 
\begin{align*}
	\eta_{t,T} := & \begin{cases}
		\frac{e^{-2\be_{t,T}^1}-e^{-\be_{t,T}^1}}{-2\be_{t,T}^1 + \be_{t,T}^1}, & e^{-2\be_{t,T}^1} \neq e^{-\be_{t,T}^1}.\\
		e^{-2\be_{t,T}^1}, & e^{-2\be_{t,T}^1} = e^{-\be_{t,T}^1}.
	\end{cases}\\
	= & \begin{cases}
		\frac{e^{-\be_{t,T}^1}-e^{-2\be_{t,T}^1}}{\be_{t,T}^1}, & \be_{t,T}^1 \neq 0,\\
		1, & \be_{t,T}^1 = 0.\end{cases}
\end{align*}
Thus, the optimal control is given by
\begin{equation}
	\label{eq:optimalControl}
	\al_t^* = \sqrt{\frac{2\la}{\delt_{t,T}\si_t^2}} \wedge 1,
\end{equation}
where
\[\delt_{t,T} = e^{-2\be_{t,T}^1} g''(X^0_T) - \be_{t,T}^2\eta_{t,T} g'(X_T^0).\]
Let
\[J(t, \mu, \al) = \frac12 g''(X_T^0) \Var_t^{\mu_1}[X^{(1/2)}_T] + g'(X_T^0)\E_t^{\mu_2}[X_T^{(1)}] + \la \int_0^T \frac{1}{\al_t}\,dt,\]
where
\[ \Var_t^{\mu_1}[X^{(1/2)}_T] =  \Var[X^{(1/2)}_T|X^{(1/2)}_t = \mu_1],\]
and similarly for $\E_t^{\mu_1}[X^{(1)}_T]$. Consider the \emph{value function} of the optimal control problem
\[V(t,\mu) = \inf_{\al\in A(L)} J(t,\mu, \al).\]
\begin{prop}
	\label{prop:alOpt}
	Assume \assref{assum:riskAndGlobExGen} and \assref{assum:gNice}.
	Then $\delt_{\blnk,T}$ is positive everywhere, $\al^*$ is Lipschitz continuous and the optimal control for the objective \eqref{eq:diffControProbO2}.
\end{prop}
\begin{proof}
	Given an initial time $t \in [0,T]$ and initial value $x\in \R$, the solution to the linear \tODE{} \eqref{eq:varx12} is given by
	\[\Var[X_T^{(1/2),t}(x)] = x e^{2\be^1_{t,T}} + \int_t^T e^{2\be^1_{t,s}} \si_s^2 \al_s \,ds, \quad x\in \R, t\leq T.\]
	Further, consider the solution $Y$ to the the backward equation \eqref{eq:BSDE} and let 
	\[\tau_\ep = 0 \vee \sup\set{t\in [0,T] : (Y_t)_1 < \ep}\]
	for any $\ep > 0$.
	Since $Y$ is continuous and $(Y_T)_1 = \frac12 g''(X_T^0) > 0$ by Assumption \assref{assum:gNice}, we have $\tau_\ep < T$ for all $\ep < \frac12 g''(X_T^0)$.
	Note that $Y$ does not depend on $\mu$ and so neither does $\tau_\ep$.
	
	Our goal now is to apply Theorem 6.4.6 in \cite{pham2009continuous} on the interval $[\tau_\ep,T]$ and conclude that $\al^*$ given in \eqref{eq:optimalControl} is an optimal control on $[\tau_\ep,T]$.
	The candidate $\al^*$ minimizes the Hamiltonian according to \eqref{eq:hamiltMin}. It remains to show that given $t\in [\tau_\ep,T]$
	the map
	
	\[\R^2 \times [0,1] \to \R, (\mu, a)\mapsto \cH(t,\mu,Y_t, a)\]
	is convex.
	Indeed, this map is in $C^2(\R^2 \times (0, 1))$ with Hessian
	\[\mat{0 & 0 & 0\\
		0 & 0 & 0\\
		0 & 0 &  2\la a^{-3}},\]
	which is positive semidefinite.
	Thus, $\al^*$ is optimal on $[\tau_\ep, T]$.
	
	Note that $X^0 \in C^1([0,T])$ and $\der b^0, \der^2 b^0, g', g'' \in C(\R), \si^2\in C^1([0,T])$. Hence, by the fundamental theorem of calculus $\be^k_{\blnk, T} \in C^1([0,T])$ for $k\in = 1,2$, and so $\al^*$ is Lipschitz continuous.
	\[(t,\mu) \mapsto \Var[X_T^{(1/2),t,\al^*}(\mu_1)]\]
	is in $C^{1,3}([0,T]\times \R)$. Similarly we can show
	\[(t,\mu) \mapsto \E[X_T^{(1),t,\al^*}(\mu_2)] \in C^{1,3}([0,T]\times \R)\]
	and $\int_\blnk^T (\al^*)^{-1}_s\,ds \in C^1([0,T])$. Hence, 
	\[V = J(\blnk, \blnk, \al^*) \in C^{1,3}([\tau_\ep,T]\times \R^2).\]
	By Theorem 6.4.7 in \cite{pham2009continuous} we can conclude that the solution of the backward equation \eqref{eq:BSDE} satisfies
	\[Y_t = \nabla_\mu V(t,\mu), t\in [\tau_\ep, T].\]
	
	Let us show $\der_{\mu_1} V(t,\mu)$ is bounded away from zero.
	With $e_1 = \mat{1 & 0}^\dagger$, we have
	\begin{align*}
		\frac{J(t, \mu + \delt e_1, \al) - J(t, \mu, \al)}{\delt} = & \frac12 g''(X_T^0) e^{2\be^1_{t,T}}.
	\end{align*}
	Therefore,
	\begin{align*}
		\der_{\mu_1} V(t, \mu) = & \lim_{\delt \to 0} \frac{\inf_{\al \in A} J(t, \mu + \delt e_1, \al) - \inf_{\al \in A} J(t, \mu, \al)}{\delt}\\
		\geq & \lim_{\delt \to 0} \frac{\inf_{\al \in A} (J(t, \mu + \delt e_1, \al) - J(t, \mu, \al))}{\delt}\\
		\geq & \frac12 g''(X_T^0) e^{2\be^1_{t,T}}\\
		> & 0.
	\end{align*}
	Set $\ep = \frac14 g''(X_T^0)\min_{t\in [0,T]} e^{2\be_{t,T}^1} > 0$. If $\tau_\ep > 0$, then 
	\[0 < \der_{\mu_1} V(t,\mu) = (Y_{\tau_\ep})_1 < \ep \leq \frac12  \der_{\mu_1} V(t,\mu),\]
	which is a contradiction.
	Hence $\tau_\ep = 0$. Therefore $(Y_\blnk)_1 = \delt_{\blnk, T}$ is positive everywhere and $\al^*$ is the optimal control on $[0,T]$.
\end{proof}

\section{Proof of the main result}
\label{sec:proofMain}
Using the our previous insights into the continuous-time theory of mini-batch SGD we can finally prove our main result.

\begin{proof}[Proof of Theorem \ref{thm:optBS}]
Firstly, Assumption \assref{assum:riskAndGlobEx} implies global unique existence of continuous solutions to \eqref{eq:gf1D} and the following family of \tSDE s
\begin{equation}
	\label{eq:SME2fmb}
	dX_t^h = - \cR'(X_t^h) - \frac h2 \cR''(X_t^h)\cR'(X_t^h)\,dt + \sqrt{h \al_t \Si(X_t^h)}\,dW_t.
\end{equation}
Setting $g := \cR$, $\si_t := \Si(X_t^0)$, $b^0 = -\cR'$ and $b^1 = -\frac12 \cR'' \cR'$ we see that it implies Assumptions \assref{assum:riskAndGlobExGen} and \assref{assum:gNice}. By Proposition \ref{prop:alOpt}, the solution to the Langrage dual to problem \eqref{eq:diffControProbO32} with Lagrange multiplier $\la > 0$ is given by $\al^*(\la)$. 
Note that by Assumption \assref{assum:riskAndGlobEx}, $\delt_{t,T}$ and $\Si(X_t^0)$ are continuous in $t$. Thus, $\al^*$ is bounded on $[0,T]$ from below, away from $0$. Hence, the dominated convergence theorem implies that
\[C : (0,\infty) \to \R, \la \mapsto \int_0^T \frac{1}{\al_t^*(\la)}\,dt \]
is continuous. We have 
\[\lim_{\la \to 0} C(\la) = \infty, \quad \lim_{\la \to \infty} C(\la) = T.\]
Hence, there exists a $\la > 0$ with $C(\la) = c$, as $c \geq T$, and then $\al^*(\la)$ is the optimum in \eqref{eq:diffControProbO32}.
By Corollary \ref{cor:funOfPerturbSeriesOptim} with $\ep = \sqrt h$ we can transfer the optimal control from the series expansion $X^0 + \sqrt h X^{(1/2)} + h X^{(1)} + h^{3/2} X^{(3/2)}$ back to the solution of \eqref{eq:SME2fmb}, and so there exists a constant $C > 0$, depending on the initial value of $X$, with
\begin{equation}
\label{eq:XquasiOptControl}
\min_{\al\in A(L)} \E \cR(X_T^{h,\al}) = \E \cR(X_T^{h,\al^*}) + C h^2 
\end{equation}
Now, Assumptions \assref{assum:riskAndGlobEx} and \assref{assum:fNice} ensure that \assref{assum:H} and \assref{assum:barHandSi} are fulfilled, uniformly in $\al\in A(L)$ (cf.\ also Remark \ref{rem:simplAssum}). Thus, we can approximate \eqref{eq:fmbSGD} by the second-order diffusion approximation \eqref{eq:SME2fmb}. In particular, Theorem \ref{thm:2ndOrderSME}, Corollary \ref{cor:2ndOrderSMEOptim} and \eqref{eq:XquasiOptControl} imply there exist constants $C_1, C_2, C_3 > 0$, depending on the shared initial value of $\chi$ and $X$, with
\begin{align*}
	\min_{\al\in A(L)}\E \cR(\chi^{h,\al}_{\floor{T/h}}) = &\min_{\al\in A(L)} \E \cR(X_T^{h,\al}) + C_1 h^2\\
	= &\E \cR(X_T^{h,\al^*}) + C_2 h^2\\
	= &\E \cR(\chi^{h,\al^*}_{\floor{T/h}}) + C_3 h^2,
\end{align*}
for all $h\in (0,1)$.
\end{proof} 

\section{Properties of the optimal volatility control for linear regression}
\label{sec:propOfOptVolLinReg}
Recall the optimal volatility control \eqref{eq:optVolLinReg} in the case of linear regression with SGD.
\subsection{Lipschitz constant}
We want to determine an upper bound on the Lipschitz constant of $\sqrt{\al^*}$. Set $s_t := \ga + e^{2\ka t}$.
Note that $\al^*$ is differentiable almost everywhere, with
\begin{align*}
	\der_t \sqrt{\al^*_t} = - \frac{\ka}{2\la} e^{2\ka t} (\al_t^*)^{5/2},
\end{align*}
for $t > \check t$, and $\der_t \sqrt{\al^*_t} = 0$ for $t\in [0,\check t)$.
Hence, we can get a bound on the the Lipschitz constant of $\sqrt{\al^*}$,
\begin{align*}
	\nrm{\sqrt{\al^*}}{\Lip} \leq & \frac{\ka}{2\la} e^{2\ka T}.
\end{align*}
Thus, in Theorem \ref{thm:optBS} we may pick any $L \geq  \frac{\ka}{2\la} e^{2\ka T}$.

\subsection{Determining the Lagrange multiplier}
We have
\begin{align*}
\int_0^T \frac{1}{\al^*_t(\la)}\,dt = & \int_0^{\check t(\la)} 1\,dt + \la^{-1/2} \int_{\check t(\la)}^T \sqrt{\ga + e^{2\ka t}}\,dt \\
= &   \check t(\la) + \la^{-1/2} (F(T) - F(\check t(\la))),
\end{align*}
where 
\[F(t) := \frac1\ka \left(\sqrt{\ga + e^{2\ka t}} - \sqrt \ga \operatorname{ArcTanh}\left(\frac{\sqrt{\ga + e^{2\ka t}}}{\sqrt \ga}\right)\right).\]
We can apply Newton's method to find a zero of $\la \mapsto \check t(\la) + \la^{-1/2} (F(T) - F(\check t(\la))) - c$. Alternatively, if $\la \leq \ga + 1$, then 
\begin{align*}
c = \int_0^T \frac{1}{\al^*_t(\la)}\,dt &\Leftrightarrow \la = \frac{(F(T)-F(0))^2}{c^2}.
\end{align*}

\section{Setup of the numerical experiment}
\label{sec:experiment}
One run of the experiment proceeds as follows.
First, we generate $N$ artificial data points according to the linear model
\[\by = -\bx + \bep,\]
where $\bx, \be \sim \cN(0,1)$ and $\bx, \be$ are independent. We fix a number of SGD steps $M$, such that $N$ is divisible by $M$. Then we use mini batch SGD to fit a linear predictor using square loss in a single epoch, with two different batch size schedules. The first schedule has constant batch size, more precisely
\[B_n^c := N/M.\]
With the second schedule, the batch size in the $n$-th step is given by
\[B_n^o = \operatorname{round}(1/\al^*_{nh}(\la)).\]
Here, $\al^*$ is the optimal volatility schedule in \eqref{eq:optVolLinReg}.
Using binary search we determine $\la$, such that
\[\sum_{n=1}^M B_n^o = N.\]
Both schedules are used $1000$ times for training, yielding instances $\hat \chi^{1,c},\dots, \hat \chi^{1000,c}$ with constant batch size and $\hat \chi^{1,o},\dots, \hat \chi^{1000,o}$ with \enquote{optimal} batch sizes.
Then, we calculate the average excess population risk
\[r^s : n\mapsto \frac{1}{1000} \sum_{i=1}^{1000} (\cR(\hat \chi_n^{i,s}) - \cR^*),\]
for $s = c, o$. Then, we re-scale time to track the number of samples processed, rather than the number of steps. That is, we plot
\[\left(\sum_{n=0}^\nu B_n^s, r^s(\nu)\right), \nu = 0,1,\dots, M,\]
for $s = c, o$.
Additionally, we superimpose the plot of the sequence of \enquote{optimal} batch sizes, in the same time scale
\[\left(\sum_{n=0}^\nu B_n^o, B_\nu\right), \nu = 0,1,\dots, M.\]
\end{alphasection}

\bibliographystyle{abbrv}
\bibliography{batchsize}

\end{document}

%% file: commands.tex
\usepackage{amsmath, amsthm, amssymb, enumerate, dsfont, comment, thmtools}
\usepackage[utf8]{inputenc}
\usepackage{csquotes}
\usepackage{tikz-cd}
\usepackage{tikz}
\usepackage{float}
\numberwithin{equation}{section} %this is so that the equation are numbered according to the section number
\usepackage{bookmark}

\newtheorem{lem}{Lemma}[section]
\newtheorem*{lem*}{Lemma}
\newtheorem{satz}[lem]{Theorem}

\newtheorem{prop}[lem]{Proposition}
\newtheorem*{prop*}{Proposition}

\newtheorem*{thm*}{Theorem}
\newtheorem{cor}[lem]{Corollary}
\newtheorem*{cor*}{Corollary}
\theoremstyle{definition}

\newtheorem*{defi*}{Definition}
\theoremstyle{remark}
\newtheorem*{example*}{Example}

\newcounter{assum}
\newenvironment{assum}{\par\noindent\refstepcounter{assum}\textbf{Assumption (A\arabic{assum})} \itshape}{\par}
\newcommand{\assref}[1]{(A\ref{#1})}

\newcommand{\Cov}{\operatorname{Cov}}

\newcommand{\innp}[2]{\langle #1, #2\rangle}

\catcode`,\active

\catcode`\,12

\declaretheorem[style=definition,sibling=lem, qed=$\diamondsuit$, name=Remark]{rem}

\newcommand{\set}[1]{\{#1\}}

\newcommand{\E}{\mathbb E}

\newcommand{\floor}[1]{\left\lfloor #1 \right\rfloor}
\newcommand{\ceil}[1]{\left\lceil #1 \right\rceil}

\newcommand{\der}{\partial}

\newcommand*{\cA}{\mathcal{A}}
\newcommand*{\cB}{\mathcal{B}}

\newcommand*{\cD}{\mathcal{D}}

\newcommand*{\cF}{\mathcal{F}}

\newcommand*{\cH}{\mathcal{H}}

\newcommand*{\cN}{\mathcal{N}}
\newcommand*{\cO}{\mathcal{O}}

\newcommand*{\cR}{\mathcal{R}}
\newcommand*{\cS}{\mathcal{S}}

\newcommand*{\cZ}{\mathcal{Z}}

\newcommand*{\N}{\mathbb{N}}
\newcommand*{\R}{\mathbb{R}}

\renewcommand*{\P}{\mathbb{P}}

\newcommand*{\om}{\omega}
\newcommand*{\Om}{\Omega}
\newcommand*{\si}{\sigma}
\newcommand*{\Si}{\Sigma}
\newcommand*{\al}{\alpha}
\newcommand*{\be}{\beta}
\newcommand*{\thet}{\theta}

\newcommand*{\delt}{\delta}
\newcommand*{\Delt}{\Delta}
\newcommand*{\ph}{\varphi}

\newcommand*{\la}{\lambda}
\newcommand*{\ka}{\kappa}
\newcommand*{\ga}{\gamma}

\newcommand*{\ep}{\varepsilon}

\newcommand{\Lip}{\operatorname{Lip}}

\newcommand*{\blnk}{\cdot}

\newcommand{\tIto}{Itô}
\newcommand{\tBm}{Brownian motion}

\newcommand*{\tme}{\mathbb T}

\newcommand*{\nrm}[2]{\|#1\|_{#2}}
\newcommand*{\Var}{\operatorname{Var}}

\DeclareMathOperator*{\amin}{\operatorname{argmin}}

\newcommand*{\tSDE}{stochastic differential equation}
\newcommand*{\tODE}{ordinary differential equation}

\newcommand*{\mat}[1]{\begin{pmatrix}#1\end{pmatrix}}
\newcommand*{\idK}{\mathds 1}